\documentclass[11pt]{article}
\usepackage{amsmath, amsfonts, amssymb, amsthm}
\usepackage{graphicx, mathrsfs, dsfont}
\usepackage{enumitem}
\usepackage{comment}
\usepackage[numbers,sort&compress]{natbib}
\usepackage{hyperref}
\usepackage{tikz,pgf}
\usepackage{pgfplots}
\usetikzlibrary{shapes.geometric}
\usetikzlibrary{arrows}
\usetikzlibrary{automata,patterns,calc}
\usepackage{tikz-3dplot}
\pgfplotsset{compat=1.18}
\usepackage[T1]{fontenc}

\usepackage{float}
\providecommand{\U}[1]{\protect\rule{.1in}{.1in}}

\newtheorem{theorem}{Theorem}[section]
\newtheorem{corollary}[theorem]{Corollary}
\newtheorem{lemma}[theorem]{Lemma}
\newtheorem{proposition}[theorem]{Proposition}
\newtheorem{definition}[theorem]{Definition}
\theoremstyle{definition}
\newtheorem{remark}[theorem]{Remark}

\newtheorem{example}[theorem]{Example}

\newcommand{\R}{\mathbb{R}}

\newcommand{\N}{\mathbb{N}}

\newcommand{\proj}{\mathrm{proj}}

\usepackage{todonotes}

\title{The reach and limits of slope eikonal equations in compact spaces}
\date{}
\author{David Salas, Sebasti\'an Tapia-Garc\'ia, Francisco Venegas M.}
\begin{document}

\maketitle

\begin{abstract}
It is a well known fact that the eikonal equation is well posed in complete length spaces. 
Among the studied notions of solutions in the literature, there is one that can be defined in \emph{any} metric space using the local (descent) slope and considering \emph{pointwise solutions}: functionals such that their slope coincides with the prescribed data at every point of the domain. In this work we explore the question ``Can we characterize the class of compact metric spaces in which every slope eikonal equation (under standard assumptions) always admits a pointwise solution?''. We provide a purely metric characterization of these spaces, as well as some interesting examples and counterexamples that illustrate the reach and limitations of the concept.  
\end{abstract}

\section{Introduction}\label{sec:intro}

Hamilton-Jacobi equations represent a paramount class of nonlinear first-order partial differential equations, 
continuing to be a fruitful field of research. Typically, a (stationary) Hamilton-Jacobi equation is written as

\[H(x,u,Du)=0 \quad\text{in }\Omega\subset\R^n,\]

\noindent with appropriate boundary conditions in $\partial \Omega$. In this setting, the viscosity theory developed by Crandall and Lions \cite{CrandallLions1983,CrandallEvansLions1984} provides a robust framework to study these equations, with a flexible notion of weak solution that allows for comparison, stability and uniqueness results. 
\\

A particular case of a Hamilton-Jacobi equation is the so-called eikonal equation, which is typically written as

\begin{equation*}
    \begin{cases}
   |D u(x)|=\ell(x) \quad&\text{for all } x\in \Omega,  \\
    u(x)  = g(x)  \quad&\text{for all } x\in \partial \Omega,
\end{cases}
\end{equation*} 
for appropriate right-hand side $\ell$ and boundary data $g$.\\

In recent years, there have been numerous efforts to develop the theory of viscosity solutions and Hamilton-Jacobi equations (and to the eikonal equation in particular) in the framework of metric spaces due to its numerous applications, see, e.g., \cite{balogh2012functional,cardaliaguet2013notes,cardaliaguet2008deterministic,feng2012hamilton,feng2013optimal,nguyen2008hamilton,gozlan2014hamilton,lasry2007mean,LT2025,lott2007hamilton,jerhaoui2024viscosity,gangbo2015existence,gigli2025viscosity}.\\

A significant challenge in this endeavor is to find a suitable replacement for the notion of derivative. In the case of the eikonal equation, this has been explored by several authors with different approaches to defining extensions of viscosity solutions:  Ambrosio and Feng in \cite{Ambrosio2014:OnAClass}; Gangbo and Święch in \cite{Gangbo2015:Metric}; and also Giga, Hamamuki and Nakayasu in \cite{Giga2015:Eikonal}.\\

More recently, in \cite{Liu2021:Equivalence}, Liu, Shanmugalingam and Zhou studied eikonal equations based on the notion of \emph{local slope} (also known as descent slope, metric slope or De Giorgi slope, first introduced in \cite{GMT1980} to study evolution equations in metric spaces), which is a purely metric object that recovers the scalar notion of the norm of the gradient. For a metric space $(X,d)$, a function $u:X\to\R\cup\{+\infty\}$ and a point $\overline x\in\mathrm{dom}\,u$, the local slope of $u$ at $\overline x$ is defined by
\begin{equation}\label{eq:Def-Slope}
    s[u](\overline{x}):=\limsup_{x\to \overline{x}}\dfrac{\max\{u(\overline{x})-u(x),0\}}{d(\overline{x},x)},
\end{equation}
under the convention that $s[u](\overline{x})=0$ if $\overline{x}$ is isolated. The local slope has been used to develop many other results in, for example, determination \cite{DaniilidisMicloSalas2024:Modulus,DaniilidisSalas2022:Determination,DaniilidisTriSalas2024:Compatibility}, stability \cite{DaniilidisDrusvyatskiy2023:Robustly,DaniilidisSalasTapia-Garcia2025:Attouch,PerezThibaultZagrodny2025:Convergence}, geometric descent \cite{Drusvyatskiy2015:Curves}, error bounds \cite{FabianEtAl:2010:Error}, Lipschitz extensions \cite{DePontiSomaglia2025:Global} and even Banach-Stone-type theorems \cite{DJV}. These results reinforce the value of the local slope as first-order information in metric spaces.
\\ 

Concretely, the approach in \cite{Liu2021:Equivalence} was to extend the notion of Monge solutions to the eikonal equation (introduced by Newcomb II and Su in \cite{newcomb1995eikonal} and further developed by Briani and Davini in \cite{BrianiDavini2005:Monge}), by considering the continuous pointwise solutions of 
\begin{align} \label{eq: eikonal introduction}
    \begin{cases}
        s[u](x)=\ell(x),\quad &\text{for all }x\in\Omega,\\
        u(x)=g(x),\quad&\text{for all }x\in \partial\Omega.
    \end{cases}
\end{align}

We call \eqref{eq: eikonal introduction} as the \emph{slope eikonal equation}. In a striking result, it is shown in \cite{Liu2021:Equivalence} that the notions of viscosity solutions to the eikonal equation mentioned so far coincide on length spaces.\\ 

Even though length spaces are a very important setting in metric analysis, one of the most appealing features of \eqref{eq: eikonal introduction} is that the notion of continuous pointwise solutions using local slopes allows to formulate the eikonal equation in arbitrary complete metric spaces. The local slope is always well-defined and, even more, the aforementioned determination results of \cite{DaniilidisMicloSalas2024:Modulus,DaniilidisSalas2022:Determination,DaniilidisTriSalas2024:Compatibility} can be seen as \emph{uniqueness} results for \eqref{eq: eikonal introduction}, valid in arbitrary metric spaces.\\

At this point a warning arises concerning slope eikonal equations: in contrast with the previously mentioned viscosity-type approaches, well-posedness (in the sense of existence and uniqueness of solutions) is not guaranteed. 
The notion used in \cite{Ambrosio2014:OnAClass, Gangbo2015:Metric} requires the metric space to be a length space for well-posedness, while the one defined in \cite{Giga2015:Eikonal} requires the existence of absolutely continuous curves connecting points in the domain of the equation with the boundary of said domain.\\

The goal of the present work is to characterize the class of metric spaces where, under some standard assumptions on $\ell$ and $g$, the slope eikonal equation \eqref{eq: eikonal introduction} is well-posed, in the sense that it admits continuous pointwise solutions. In this first study, we focus on the case of compact metric spaces, considering only equations where the prescribed slope $\ell(\cdot)$ is continuous and bounded, with $\inf_{\Omega}\ell >0$. Even in this setting, it is not clear how to characterize these spaces. Nevertheless, such a characterization delineates the reach and limits of slope eikonal equations and, in some sense, of the local slope as first-order information.


\subsection{A simple counterexample}

One might hope that, for every metric space, equation~\eqref{eq: eikonal introduction} admits solutions, at least under some suitable conditions. Before we continue, let us exhibit a first example where no eikonal equation admits solutions.\\

Let $X=[0,1]$ and let $\Omega$ be any proper  open subset of $X$. Endow $X$ with the snowflake distance $d(x,y) = \sqrt{|x-y|}$. Set $g:\partial\Omega\to \R$ as $g\equiv 0$. Suppose that with this boundary data, there exists a positive function $\ell:\Omega\to (0,+\infty)$ such that equation~\eqref{eq: eikonal introduction} admits a continuous solution $u:\overline{\Omega}\to \R$. Note that $u$ must be strictly positive on $\Omega$ since it cannot have a global minimum there. By construction, $u$ can be continuously extended to $X$ by setting $u(x) = 0$ for all $x\in X\setminus \overline{\Omega}$, as consequence, $s[u](x) = 0$ on $X\setminus \Omega$.  Thus, we can apply the following determination result (which si directly deduced from \cite{DaniilidisSalas2022:Determination}) valid in any compact metric space:

\begin{theorem}[Determination Theorem,~{\cite[Theorem 2.4]{DaniilidisSalas2022:Determination}}] Let $(X,d)$ be any  compact metric space, and let $u_1,u_2:X\to \R$ be two continuous functions such that
	\begin{itemize}
		\item $s[u_1](x) = s[u_2](x) <+\infty$ for all $x\in X$; and
		\item For all $x\in X$ with $s[u_1](x) = s[u_2](x) = 0$, $u_1(x) = u_2(x)$.
	\end{itemize}
Then, $u_1 = u_2$.	
\end{theorem}

The theorem above entails that $u$ must be the unique continuous function verifying that $s[u] = \ell$ in $\Omega$ and $u = 0$ in $X\setminus\Omega$. Now, choose $0<a<b<1$ such that $[a,b]\subset \Omega$ and let us define $\psi:[0,1]\to \R$ by
\[
\psi(x) = \begin{cases}
	x-a &x\in [a, (a+b)/2],\\
	b-x & x\in [(a+b)/2, b],\\
	0&\text{ otherwise}.
\end{cases}
\] 
It is easy to verify that $\psi$ is continuous and that for every $x\in (a,b)$, one has that
\[
s[\psi](x) = \limsup_{y\to x} \frac{[\psi(x)-\psi(y)]_+}{\sqrt{|x-y|}} =  \limsup_{y\to x} \frac{|x-y|}{\sqrt{|x-y|}} = 0.
\]
Thus, $\psi$ is a nonconstant function with zero slope everywhere (in fact, $\psi$ is an example of a \emph{locally flat} function, see for instance~\cite{Weaver}). One can readily verify that
\[
s[u+\psi](x) = s[u](x),\quad\text{for all }x\in X,
\] 

and so, since $u+\psi \equiv 0$ on $X\setminus\Omega$, we get a contradiction with the determination result of \cite{DaniilidisSalas2022:Determination}. Thus, in this metric space, no eikonal equation as~\eqref{eq: eikonal introduction} admits solutions. Moreover, every Lipschitz function in this space must have zero slope in a dense set, making the determination result of \cite{DaniilidisSalas2022:Determination} trivial. Here, the local slope fails to provide meaningful first-order information.

\subsection{Problem formulation and main contribution}

Our main goal is to characterize the class of compact metric spaces for which slope eikonal equations admit continuous pointwise solutions, which is a proper subclass of all compact metric spaces, as we have just shown.

Of course, we need to consider some compatibility condition between the prescribed slope $\ell$ in $\Omega$ and the Dirichlet data $g$ on $\partial \Omega$. Motivated by the classical setting of euclidean spaces, together with the developments in length spaces, we propose to consider the following: For $\Omega\subsetneq X$ a nonempty open set, we say that a pair of bounded continuous functions $\ell:\Omega\subset X\to \R$ and $g:\partial \Omega \to \R$, satisfies the compatibility condition \eqref{CC} if for every $x,y\in \partial \Omega$,
\begin{align} \tag{CC}\label{CC}
    g(x)-g(y)\leq \inf \left\{\int_0^T\ell(\gamma(t))dt \right\}
\end{align}
where the infimum is taken over the family of $1$-Lipschitz curves $\gamma:[0,T]\to \overline{\Omega}$ such that $\gamma(0)=y$, $\gamma(T)=x$ and $\gamma((0,T))\subset\Omega$, under the convention $\inf \emptyset =+\infty$. With this definition, we are ready to properly define eikonal spaces.
\begin{definition}
    A metric space $(X,d)$ is called an eikonal space if, for any nonempty open set $\Omega\subsetneq X$ and bounded continuous maps $\ell:\Omega\to \R$, with $\inf \ell>0$, and $g:\partial \Omega\to\R$, satisfying~\eqref{CC}, the eikonal equation
    \begin{equation} \label{eq: eikonal complete}
    \begin{cases}
        s[u](x)=\ell(x),\quad &\text{for all }x\in\Omega,\\
        u(x)=g(x),\quad&\text{for all }x\in \partial\Omega,
    \end{cases}
    \end{equation}
    admits a continuous pointwise solution. 
\end{definition}
We would like to point out that \textit{a priori} there are no further assumptions on the metric space $X$ besides the solvability of equation~\eqref{eq: eikonal complete}. Indeed, if there are no rectifiable curves connecting the points of $x,y\in \partial \Omega$, then the right-hand side of~\eqref{CC} is simply $+\infty$. One might think that the compatibility condition~\eqref{CC} imposes implicitly some level of rectifiability on the space, since otherwise too many pairs $(\ell,g)$ could fit the condition by vacuity.  Surprisingly, as a consequence of a result in \cite{Drusvyatskiy2015:Curves}, in the next section we show that if a compact metric space admits enough functions with continuous local slope, then it must be rectifiably connected, see Lemma~\ref{lax formula} and Corollary~\ref{cor: rec-connected}. Thus, \textit{a posteriori}, the compatibility condition~\eqref{CC} is not imposing any supplementary structure on eikonal spaces, but it is rather the appropriate condition to study slope eikonal equations with continuous data. \\

The main result of this work is the characterization of compact eikonal spaces by considerably reducing the family of equations that must admit a solution. 

\begin{theorem}\label{thm: main} Let $(X,d)$ be a compact metric space such that for every nonempty closed set $K\subset X$, the equation
\begin{equation}\label{eq: eikonal 0}
    \begin{cases}
        s[u](x)=1,\quad &\text{for all }x\in X\setminus K,\\
        u(x)=0,\quad&\text{for all }x\in K.
    \end{cases}
    \end{equation}
admits a continuous pointwise solution. Then, $(X,d)$ is an eikonal space.
\end{theorem}
Noting that the solution of the above equation must be the intrinsic distance to the set $K$, see Lemma~\ref{lax formula}, we have the following consequence.
\begin{corollary}\label{cor:intrinsic-distance-to-K}
    A compact metric space $(X,d)$ is eikonal if and only if, for every nonempty closed set $K\subset X$, we have that $s[d_I(\cdot,K)]\equiv1$ on $X\setminus K$. 
\end{corollary}

\paragraph{Outline of the manuscript:} After short preliminaries in Section~\ref{sec:pre}, we show in Section~\ref{sec:representation-formula} that the existence of functions with continuous slope relates to rectifiable connectivity of the space, as well as representation formula for the (unique) candidate of solution of the slope eikonal equation. In Section~\ref{sec:examples}, we provide several examples and counterexamples narrowing down what eikonal spaces must be. In particular, we show that this notion is not stable by bi-Lipschitz equivalence and we provide an example of a rectifiably connected space bi-Lipschitz equivalent to a geodesic space which fails to admit solutions for any slope eikonal equation with continuous data. We prove our main theorem in Section~\ref{sec:main}. We finish the work in Section~\ref{sec:Comparison} by comparing our results with the theory of viscosity solutions of eikonal equations previously developed in \cite{Giga2015:Eikonal,Gangbo2015:Metric,Ambrosio2014:OnAClass,Gangbo2014:Optimal,Liu2021:Equivalence}, and presenting some perspectives. 

\section{Preliminaries}\label{sec:pre}
 For a metric space $(X,d)$, we denote by $\mathcal{C}_b(X)$ and $\mathcal{C}(X)$ the spaces of real valued bounded continuous function and continuous functions respectively.  For $T>0$, we denote by $\mathrm{Lip}_1([0,T],X)$ and $AC([0,T],X)$ the spaces of $1$-Lipschitz curves and absolutely continuous curves from $[0,T]$ to $X$, respectively.
 For a continuous curve $\gamma:[0,T]\to X$, with $T>0$, its \emph{length} is defined by
     \[\mathrm{len}(\gamma):= \sup\left\{\sum_{i=0}^{n-1} d\left(\gamma(t_i),\gamma(t_{i+1})\right)~:~0=t_0\leq t_1\leq \ldots\leq t_n=T\right\}\in [0,+\infty],\]
     and $\gamma$ is said to be \emph{rectifiable} if $\mathrm{len}(\gamma)<+\infty$. A metric space where every pair of points can be connected through a rectifiable curve is called a \emph{rectifiably connected} space.\\
     
     The \emph{intrinsic distance}, ${d_I:X\times X\to [0,+\infty]}$, is defined by 
    \[ d_I(x,y)=\inf\left\{\mathrm{len}(\gamma)~|~\gamma:[0,T]\to X\text{ is a continuous curve, }\gamma(0)=x \text{ and }\gamma(T)=y\right\}.\]
    Here we use again the convention $\inf \emptyset =+\infty$.
    The intrinsic distance $d_I$ defines a metric on $X$ if and only if $(X,d)$ is rectifiably connected.
    Note that $d\leq d_I$ for any metric space. Recall that a metric space is said to be a length space if $d = d_I$.\\

We close this section with a dynamic programming principle, which can be found \cite[Lemma~4.1]{Giga2015:Eikonal}. The arguments to prove this lemma follow as in the classical theory on $\R^n$.

\begin{lemma}[Dynamic Programming Principle]\label{DPP}
     Let $(X,d)$ be a metric space and let $\Omega\subset X$ be an open set such that $\partial \Omega\neq\emptyset$. 
Let $\ell \in \mathcal{C}(\Omega)$, with $\ell >0$, and $g\in \mathcal{C}(\partial \Omega)$. Assume that the value function $V:\Omega\to\R$ defined by
\[V(x):=\inf\left\{ 
     \int_0^T \ell(\gamma(s))ds+g(\gamma(T))\,\colon \begin{array}{l}
     ~ T>0,~\gamma\in \mathrm{Lip}_1([0,T],\overline{\Omega}),\\
     \gamma(0)=x,~\gamma(T)\in \partial \Omega,~\gamma([0,T))\subset\Omega
     \end{array}
     \right\}\]
is finitely-valued. Then, for all $x\in \Omega$,
\[V(x)=\inf\left\{\int_0^T\ell(\gamma(s))ds+ V(\gamma(T)):~T>0,~\gamma\in \mathrm{Lip}_1([0,T],\Omega) \right\}.\]
\end{lemma}

\section{Representation formula in compact spaces}\label{sec:representation-formula}

In the classical setting, Hamilton-Jacobi equations are often associated to optimal control problems.
Indeed, it turns out that the value function for this optimal control problem happens to be a natural candidate for (viscosity) solution of the Hamilton-Jacobi equation, see e.g. \cite{Barles2013}.
In this line, our next result states that, for an open subset $\Omega$ of a compact metric space, if there is a continuous function $f:\overline{\Omega}\to\R$ with continuous descent slope $s[f]$ on $\Omega$, then the space admits enough rectifiable curves so that $f$ can be described as a value function of an optimal control problem.\\

\begin{lemma}\label{lax formula}
    Let $(X,d)$ be a compact metric space and let $\Omega\subset X$ be nonempty open set. 
    Let $f:\overline{\Omega}\to\R$ be a continuous function such that $s[f]:\Omega \to \R$ is continuous, with $\inf_{\Omega} s[f]>0$.
    Then, $\partial \Omega \neq \emptyset$ and for all $x\in \Omega$ we have 
    \begin{equation} \label{eq:Lax-pancho}
     f(x)=\inf\left\{ 
     \int_0^T \ell(\gamma(s))ds+f(\gamma(T))\,\colon \begin{array}{l}
     ~ T>0,~\gamma\in \mathrm{Lip}_1([0,T],\overline{\Omega}),\\
     \gamma(0)=x,~\gamma(T)\in \partial \Omega,~\gamma([0,T))\subset\Omega
     \end{array}
     \right\},
    \end{equation}
    where $\ell:=s[f]$.
    Moreover, for all $x\in \Omega$, there is a $1$-Lipschitz curve $\gamma:[0,T]\to \overline{\Omega}$ that attains the above infimum.
\end{lemma}
In order to prove the above theorem, we will use \cite[Theorem 3.5]{Drusvyatskiy2015:Curves}. However, for the sake of clarity, instead of writing down the mentioned Theorem in its full generality, the statement below is an straightforward adaptation which is enough for our purposes. 
\begin{theorem}\label{thm: Drus}
    Let $(X,d)$ be a compact geodesic space and let $f:X\to\R\cup\{+\infty\}$ be a lower semicontinuous function. Assume that $f$ is continuous and $s[f]$ is lower semicontinuous on $\mathrm{dom}\,f$. 
    Then, for any $x\in \mathrm{dom}\,f$ with $s[f](x)>0$, there is a $1$-Lipschitz curve $\gamma:[0,L]\to X$ such that $\gamma([0,L])\subset \mathrm{dom}\,f$
    \[-(f\circ\gamma)'(t) = s[f](\gamma(t)),\quad\text{for a.e. }t\in[0,L].\]
\end{theorem}
\begin{remark}
    In the context of Theorem~\ref{thm: Drus}, if $\overline{x}\in \mathrm{dom}\,f$ is such that $s[f](\overline{x})>0$, then we deduce that there is a non-constant $1$-Lipschitz curve $\gamma:[0,L]\to \mathrm{dom}\, f$ such that $\gamma(0)=x$. In fact, this same conclusion can be derived from \cite[Theorem 3.5]{Drusvyatskiy2015:Curves} using only as hypothesis that the slope $s[f]$ is bounded away from zero near $x$, and dropping the requirement of lower semicontinuity. In this case, the curve won't necessarily verify the steepest descent relation $-(f\circ\gamma)'(t) = s[f](\gamma(t))$, but it will only be a realiable descent curve \cite[Definition 2.18]{Drusvyatskiy2015:Curves}.  
    We would like to point out that the construction of such a curve relays only on the compactness of the subjacent space and the definition of the local slope, as we will see in the forthcoming proof.

\end{remark}
\begin{proof}[Proof of Lemma~\ref{lax formula}]
First, we show that $\partial \Omega \neq\emptyset$. Reasoning towards a contradiction, assume that $\Omega$ is a clopen subset of $X$. 
Therefore, $\Omega$ is compact and $f$ admits a local minimum $\overline{x}\in \Omega$. So, $s[f](\overline{x})=0$, which is a contradiction.\\

Set $V:\Omega\to \R$ the function defined by the right hand side of \eqref{eq:Lax-pancho}. Recall that we use the convention $\inf \emptyset =+\infty$. \\

$f\leq V$: Let $x\in \Omega$ and assume that $V(x)<+\infty$. 
Therefore, there are curves $\gamma\in \mathrm{Lip}_1([0,T],X)$ such that $\gamma(0)=x$, $\gamma(T)\in \partial \Omega$ and $\gamma([0,T))\subset\Omega$. 
Let $\gamma\in \mathrm{Lip}_1([0,T],X)$ be such a curve.
It easily follows that $s[f\circ \gamma](t)\leq s[f](\gamma(t))= \ell(\gamma(t))$ for all $t\in[0,T)$.
Therefore
\begin{align*}
    f(x)-f(\gamma(T))\leq \int_0^T \ell(\gamma(t))ds.
\end{align*}
Since $\gamma$ is arbitrary, we deduce that $f(x)\leq V(x)$.\\

$f\geq V$: The idea is to use Theorem~\ref{thm: Drus}. Consider the Fr\'echet-Kuratowski isometric embedding (see e.g. \cite{burago2001course}) $i:X\to \ell^\infty(\N)$ into the Banach space of bounded sequences. 
Set $\mathcal{X}:=\overline{\mathrm{conv}}(i(X))$ and $\rho$ the distance on $\mathcal{X}$ inherited from $\ell^\infty(\N)$. 
Thus, $(\mathcal{X},\rho)$ is a compact geodesic space that contains an isometric copy of $X$. 
Define $F:\mathcal{X}\to\R\cup\{+\infty\}$ by $F\equiv f$ on $i(X)$ and $F\equiv +\infty$ on $\mathcal{X}\setminus i(X)$.
It directly follows that $F$ is lower semicontinuous, $F$ is continuous on $\mathrm{dom}(F)= i(X)$ and $s[F](i(x))=s[f](x)$ for all $x\in X$.\\

Let $\overline{x}\in \Omega$. Then we can use Theorem~\ref{thm: Drus} on $F$ at $\overline{x}$ to get a curve $\nu\in \mathrm{Lip}_1([0,L],\mathcal{X})$ such that $\nu(0)=\overline{x}$, $\nu([0,L])\subset \mathrm{dom}(F)=i(X)$ and 
\[
-(F\circ\gamma)'(t) = s[F](\gamma(t)),\quad\text{for a.e. }t\in[0,L].
\]
So, we deduce that
\[F(\nu(s))-F(\nu(t))=\int_s^ts[F](\nu(r))dt,\quad \text{for all }0\leq s<t\leq L.\]
Consider $\gamma:=i^{-1}\circ \nu$, we obtain that $\gamma$ is $1$-Lipschitz, $\gamma(0)=\overline{x}$, $\gamma([0,L])\subset X$ and 
\[f(\gamma(s))-f(\gamma(t))=\int_s^ts[f](\gamma(r))dt,\quad \text{for all }0\leq s<t\leq L.\]
If $\gamma([0,L])\not\subset \Omega$, then choose $T:=\inf\{t>0:~\gamma(t)\in \partial\Omega\}$. 
If $\gamma([0,L])\subset \Omega$, then we can repeat the above argument, this time starting from $\gamma(L)$, to extend the domain of definition of $\gamma$.
Indeed, since $\inf_{\Omega} s[f]>0$, a standard procedure invoking Zorn's Lemma implies the existence of a $1$-Lipschitz curve $\gamma:[0,T]\to X$, with $T<+\infty$ such that $\gamma(0)=\overline{x}$, $\gamma([0,T))\subset \Omega)$, $\gamma(T)\in \partial \Omega$ and 
\[f(\gamma(s))-f(\gamma(t))=\int_s^ts[f](\gamma(r))dt,\quad \text{for all }0\leq s<t\leq T.\]
In particular, we have that
\[V(\overline{x})\leq f(\gamma(T))+\int_0^T s[f](\gamma(r))dt= f(\overline{x}).\]
Therefore, $V(\overline{x})=f(\overline{x})$  and the infimum that defines $V(\overline{x})$ is a minimum.
\end{proof}
We have the following direct consequences.
\begin{corollary}\label{cor: rec-connected}
    Let $(X,d)$ be a compact eikonal space. Then $X$ is rectifiable connected. 
\end{corollary}
\begin{proof}
    Assume that $X$ contains at least two points.
    Let $\overline{x}\in X$ and set $\Omega= X\setminus \{\overline{x}\}$. 
    Let us show first that $\overline{x}$ is not isolated. 
    Reasoning by absurd, assume that $\Omega= X\setminus \{\overline{x}\}$ is clopen (thus compact). 
    Since $X$ is an eikonal space, there is a continuous function $u:\Omega\to \R$ such that
    \[s[u](x)=1,\quad\text{for all }x\in \Omega.\]
    Since $\Omega$ is compact, $u$ attains its minimum on $\Omega$. Therefore, there is a point $\overline{z}\in \Omega$ such that $s[u](\overline{z})=0$, a contradiction.
    So, we have that $\partial \Omega= \{\overline{x}\}$. 
    Again since $X$ is an eikonal space, there is a continuous function $u:X\to\R$ such that
    \[\begin{cases}
        s[u](x)=1,\quad &\text{for all }x\in\Omega,\\
        u(\overline{x})=0.\quad&
    \end{cases}\]
    Thus, thanks to Lemma~\ref{lax formula}, for every $x\in X\setminus\{\bar{x}\}$ there is a $1$-Lipschitz curve connecting $x$ with $\overline{x}$. 
\end{proof}

\begin{corollary}[Uniqueness of continuous pointwise solutions]\label{cor: uniqueness}
Equation~\eqref{eq: eikonal complete} admits at most one continuous pointwise solution provided $\ell\in \mathcal{C}(\Omega)$, $\inf\,\ell>0$ and $g\in \mathcal{C}(\partial \Omega)$. Moreover, the continuous pointwise solution $u:\overline{\Omega}\to\R$ is defined by
\begin{equation}\label{eq:OptimalControl-formula} 
     u(x)=\inf\left\{ \int_0^T \ell(\gamma(s))ds+g(\gamma(T)):~ \begin{array}{l}T>0,~\gamma\in \mathrm{Lip}_1([0,T],X),\\
     ~\gamma(0)=x\text{ and }\gamma(T)\in \partial \Omega \end{array}\right\}.
    \end{equation}
\end{corollary}
\begin{proof}
Let $(X,d)$ be a compact metric space and $\Omega\subset X$ be a nonempty open set.
    Let $\ell\in \mathcal{C}(\Omega)$, $\inf\,\ell>0$ and $g\in \mathcal{C}(\partial \Omega)$ be such that Equation~\eqref{eq: eikonal complete} admits at least a continuous solution.
    The representation formula provided by Lemma~\ref{lax formula} yields the uniqueness of solution.
\end{proof}
For the next corollary we recall that a metric space is called proper if every bounded and closed set is compact. 
\begin{corollary}
    Let $(X,d)$ be a proper metric space and let $f:X\to\R$ be a continuous function. 
    Assume that $s[f]:X\to \R$ is continuous. 
    Then, for any $x\in X$ such that $s[f](x)\neq 0$, there is a curve $\gamma\in \mathrm{Lip}_1([0,T],X)$, with $T>0$, such that  $\gamma(0)=x$ and
    \[
    f( \gamma(t))= f(x)- \int_0^t s[f](\gamma(s))ds,\quad\text{for all }t\in [0,T].
    \]  
\end{corollary}

We end this section by showing that Lemma~\ref{lax formula} does not hold in the noncompact setting.
  
\begin{example}
     Consider the Banach space $c_0(\N)$ of real null sequences endowed with its canonical norm, $\|x\|_\infty:=\sup_{n\in\N} |x_n|$. 
Denote by $\{e_n\}_n\subset c_0 (\N)$ its canonical basis, i.e. the vectors such that $e_{n,m}=1$ if $n=m$ and $e_{n,m}=0$ otherwise. Let $n\geq 2$ and $\Gamma_n\subset c_0(\N)$ be the set defined by
    \[\Gamma_n:=\left\{t(e_1+e_n):t\in\left[0,\frac{1}{2}\right)\right \}\cup \left\{te_1+(1-t)e_n:t\in\left[\frac{1}{2},1-\frac{1}{n}\right)\right \}.\]

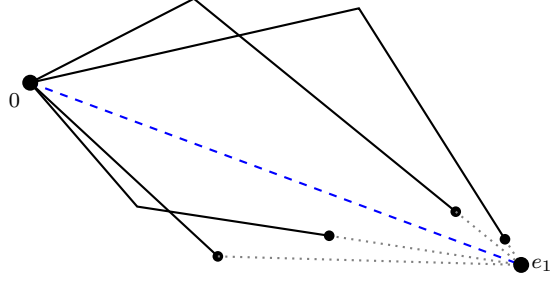
\begin{figure}[H]\label{fig_3Drasca}
    \centering
    \caption{Representation of $\bigcup_{i=1}^4\Gamma_i$.} 
\tdplotsetmaincoords{50}{120}

\begin{tikzpicture}[tdplot_main_coords, scale=5]

\coordinate (O) at (0,0,0);

\draw[thick] (0,0,0) -- (0,0.5,0.5)--(0, 0.5+0.8, 0.5 -0.8*0.5);
\fill (0,0.5+0.8,0.5 -0.8*0.5) circle (0.4pt);
\draw[gray,dotted,thick] (0,0.5+0.8,0.5 -0.8*0.5) -- (0,1.5,0);

\draw[thick] (0,0,0) --(0.3,0.5,0) -- (0.15,1,0);
\draw[gray,dotted,thick] (0.15,1,0) -- (0,1.5,0);
\fill (0.15,1,0) circle (0.4pt);

\draw[thick] (0,0,0) --(-0.3,0.4,-0.65);
\fill (-0.3,0.4,-0.65) circle (0.4pt);
\draw[gray,dotted,thick] (-0.3,0.4,-0.65) -- (0,1.5,0);

\draw[thick] (0,0,0) --(-0.7,0.6,0) -- (-0.7+0.9*0.7, 0.6 + 0.9*0.9,0);
\draw[gray,dotted,thick] (-0.7+0.9*0.7, 0.6 + 0.9*0.9,0) -- (0,1.5,0);
\fill (-0.7+0.9*0.7,0.6 + 0.9*0.9,0) circle (0.4pt);

\draw[dashed,blue,thick] (0,0,0)--(0,1+0.5,0);
\fill (O) circle (0.6pt) node[below left]{\scriptsize $0$};
\fill (0,1+0.5,0) circle (0.6pt) node[right]{\scriptsize $e_1$};
\end{tikzpicture}
\end{figure}
    
    And consider the set
    \[X:=\bigcup_{n\geq 2}\Gamma_n \cup \{e_1\}.\]
    Endow $X$ with the metric $d$ inherited by $c_0(\N)$, i.e. $d(x,y)=\|x-y\|_\infty$ for all $x,y\in X$. Then, $(X,d)$ admits a pointwise solution of an eikonal equation that is not representable by the formula given in Lemma~\ref{lax formula}.
    Indeed, consider the function $u:X\to\R$ defined by $u(x)=d(0,x)$.
    It easily follows that
    \[\begin{cases}
        s[u](x)=1,\quad& \text{for all }x\in \Omega:= X\setminus\{0\}.\\
        u(0)=0,\quad&\partial\Omega:=\{0\}.
    \end{cases}\]
    Indeed, $\Gamma_n$ are geodesics in $c_0(\N)$  joining $0$ with $(1-\frac{1}{n})e_1 + \frac{1}{n} e_n$ and
    \[1=\mathrm{Lip}(u)\geq s[u](e_1)\geq \lim_{n\to\infty}\dfrac{u(e_1)-u((1-n^{-1})e_1+n^{-1}e_n)}{d(e_1,(1-n^{-1})e_1+n^{-1}e_n)}=1.\]
    However, since there is no curve connecting $0$ and $e_1$, the value of $u(e_1)$ cannot be obtained by the formula~\eqref{eq:Lax-pancho} given in Lemma~\ref{lax formula}.
    \hfill$\Diamond$
\end{example}


\section{Examples and counterexamples}\label{sec:examples}
This section is devoted to presenting examples of compact eikonal spaces, as well as examples where the property fails. To gain further insight on this concept, we explore notions that one might expect to be related to the solvability of slope eikonal equations, which will help us narrow down the search for a metric characterization of eikonal spaces.\\

We take as a starting point the well established fact that the slope eikonal equation~\eqref{eq: eikonal complete} admits a continuous solution whenever the space $X$ is a length space. When studying equation~\eqref{eq: eikonal complete} on a rectifiably connected space $(X,d)$, one might be tempted to replace the metric $d$ with its intrinsic version $d_I$, which automatically turns $(X,d_I)$ into a length space, warrantying the existence of solutions to any version of equation~\eqref{eq: eikonal complete}. This approach has a clear downside, as some eikonal equations may not admit a solution on $(X,d)$ at all. This change of metric fundamentally alters the behavior of the eikonal equation, even when the identity mapping $\mathrm{Id}:(X,d)\to (X,d_I)$ is bi-Lipschitz. This can be seen on forthcoming Example~\ref{pato_simple}.\\ 

Before we continue, let us recall that for $C\geq 1$, a metric space $(X,d)$ is said to be a $C$-quasiconvex space if $d_I\leq Cd$. A space $(X,d)$ is said to be quasiconvex if there exists $C\geq 1$ such that it is $C$-quasiconvex. Note that a $X$ is quasiconvex if and only if $(X,d)$ and $(X,d_I)$ are bi-Lipschitz equivalent.\\

Our first example shows that being a length space is not a necessary condition for a space to be eikonal.
\begin{example}\label{circle}
    Consider the unit circle $\mathbb S^1\subset \R^2$ endowed with the distance induced by the canonical Euclidean structure of $\R^2$. 
    The intrinsic distance between any pair of points $x$ and $y$ in $\mathbb S^1$ coincides with the angle $\theta$ between them, while the Euclidean distance is $d(x,y)=2\sin\left(\tfrac{\theta}{2}\right)$. Using this, it can be shown that $\mathbb S^1$ is a $\tfrac{\pi}{2}$-quasiconvex space. Moreover, it can be shown that the identity mapping $\mathrm{Id}:(\mathbb S^1,d)\to (\mathbb S^1,d_I)$ and its inverse $\mathrm{Id}^{-1}$ have pointwise Lipschitz constant equal to $1$ at every point in $\mathbb S^1$. This is known in the literature as a \emph{pointwise isometry} (see \cite{DJ} for details). 
    It is not hard to show that any metric space that is pointwise isometric to itself endowed with its intrinsic distance is eikonal as well. Thus, $(\mathbb S^1,d)$ is an example of a quasiconvex and compact eikonal space that fails to be length, showing that the latter is not a necessary condition for the solvability of slope eikonal equations.
    \hfill$\Diamond$
\end{example}
Although the property of being pointwise isometric to a length space implies being an eikonal space, it is not necessary, as the following example shows.

\begin{example}\label{araña_no_QC}
    The following metric space is constructed as a subset of the Banach space $\ell^1(\N)$ of summable sequences. Let $(e_n)_n\subset\ell^1(\N)$ be its canonical basis. 
    For two vectors $x,y\in \ell^1(\N)$ we denote by $[x,y]$ the line segment joining $x$ and $y$, that is, $[x,y]:=\{\lambda x+(1-\lambda )y:~\lambda\in[0,1]\}$.
    For every $n\in \N$, set
    \[P_n:=\left[0,\dfrac{e_{2n}}{2n^2}+\dfrac{e_{2n+1}}{n}\right]\cup\left[\dfrac{e_{2n}}{2n^2}+\dfrac{e_{2n+1}}{n},\dfrac{e_{2n}}{n^2}\right] \subset \ell^1(\N).\]
    Consider $X:=\bigcup_{n\in \N} P_n$ and endow $X$ with the metric $d$ inherited by $\ell^1(\N)$. It follows readily that $(X,d)$ is compact.\\

\begin{figure}[H]\label{fig_3Daraña}
    \centering
    \caption{Representation of $P_1\cup P_2\cup P_3$.} 
\tdplotsetmaincoords{62}{120}

\begin{tikzpicture}[tdplot_main_coords, scale=4]

\coordinate (O) at (0,0,0);
\fill (O) circle (0.6pt);

\coordinate (A) at (0,0,0);
\coordinate (B) at (0.8,0,1);
\coordinate (C) at (1.6,0,0);

\draw[thick] (A) -- (B) -- (C);
\draw[dashed,blue] (0.8,0,0) -- (B) node[midway,right]{$e_3$};

\draw[dashed,red,-stealth] (A) -- (2,0,0) node[below]{$e_2$};

\coordinate (B) at (0.5,0.5,0.6);
\coordinate (C) at (1,1,0);

\draw[thick] (A) -- (B) -- (C);
\draw[dashed,red,-stealth] (A) -- (1.3,1.3,0) node[below]{$e_4$};
\draw[dashed,blue] (0.5,0.5,0) -- (B) node[midway, left]{$e_5$};

\coordinate (B) at (-0.15,-0.05,0.3);
\coordinate (C) at (-0.3,-0.1,0);

\draw[thick] (A) -- (B) -- (C);
\draw[dashed,red,-stealth] (A) -- (-1.2,-0.4,0) node[above]{$e_6$};
\draw[dashed,blue] (-0.15,-0.05,0) -- (B);
\end{tikzpicture}

\end{figure}
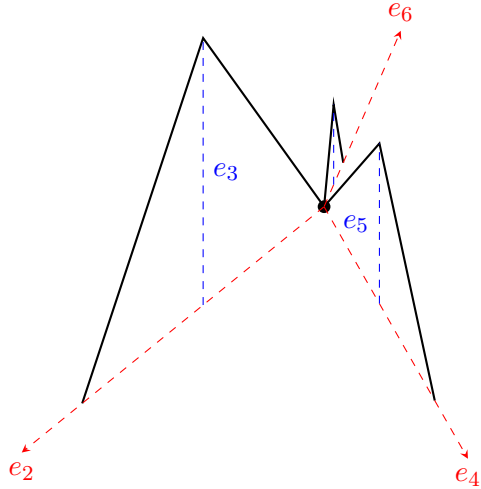
It is clear that $X$ is rectifiably connected and that $d_I(0,e_{2n}/n^2)= n^{-2}+2n^{-1}= (1+2n)d(0,e_{2n}/n^2)$ for every $n\in\N$. Therefore, $X$ is not quasiconvex (and in particular, not pointwise isometric to a length space).\\ 

We now prove that $X$ is an eikonal space using Theorem~\ref{thm: main}. Let $\Omega\subsetneq X$ be a nonempty open subset, and consider $\ell\equiv 1$. We must show that the function $u$ defined in Corollary~\ref{cor: uniqueness}, with $g=0$ on $\partial \Omega$, is a solution of
    \[\begin{cases}
        s[u](x)=1,&\quad \text{in }\Omega\\
        u(x)=0,&\quad \text{in }\partial\Omega
    \end{cases}.\]

    If $0\notin \Omega$, the equation is solved by studying the eikonal equation on each $P_n\cap \overline{\Omega}$ independently. On the other hand, if $0\in \Omega$, take $r>0$ such that $B(0,r)\subset \Omega$. So, there must exist $N\in \N$ such that $P_n\subset \Omega$ for all $n\geq N$. Hence, $\partial\Omega$ intersects only finitely many sets from the family $\{P_n\}_n$. In this case we solve first the equation restricted to $\bigcup_{k=1}^N P_k$, and then extend $u$ to each $P_n$ with $n\geq N$ simply as $u(x) = u(0) + d_I(x,0)$ for each $x\in P_n$. In both cases, the continuity of $u$ and the computation of $s[u]$ are direct.\hfill $\Diamond$
\end{example}
With this example, we have shown that quasiconvexity is not a necessary condition for the eikonal property. 
Our next example, which is also constructed inside the space $\ell^1(\N)$, is in fact quasiconvex, but nevertheless fails the eikonal property at a prescribed point, in the sense that equation~\eqref{eq: eikonal complete} has no solution whenever the prescribed point is in $\Omega$. This shows that quasiconvexity is not a sufficient condition either. We present the construction in Example~\ref{pato_simple}, and prove our statement in Proposition~\ref{prop:Pato-simple-properties}.

\begin{example}\label{pato_simple}
   Consider the canonical basis $\{e_n\}_n\subset \ell^1(\N)$. For $n\in \N$, consider the sets $S_n\subset \ell^1(\N)$ given by
\[S_n:=\left\{te_1+\alpha(t)e_{n+1}:~t\in[0,1]\right\},\]
where $\alpha(t):=t$ if $t\in \left[0,\frac{1}{2}\right]$ and $\alpha(t)=1-t$ if $t\in \left[\frac{1}{2},1\right]$.
Set $x_n:= (\frac{1}{2}-\frac{1}{2^{n}})e_1$ for all $n\in\N$ and $x_\infty:= \tfrac{1}{2}e_1$.
Define 
\[X:=\bigcup_{n\geq 1}\left\{\dfrac{1}{2^{n+1}}S_n+x_{n}\right\}\cup\{x_\infty\}.\]

\begin{figure}[H]\label{fig_3Dpato}
    \centering
    \caption{Representation of $\displaystyle\bigcup_{i=1}^{4}\left\{\dfrac{1}{2^{i+1}}S_i+x_{i}\right\}\cup\{x_\infty\}.$}

\pgfplotsset{compat=1.18}

\begin{tikzpicture}
  \begin{axis}[
    view={120}{30}, 
    axis lines=middle,
    xlabel={$e_2$},
    ylabel={$e_1$},
    zlabel={$e_3$},
    xmin=0, xmax=0.14,   
    ymin=0, ymax=0.7,
    zmin=0, zmax=0.07,
    xtick=\empty, ytick=\empty, ztick=\empty
  ]

    \addplot3[thick,blue] coordinates {
      (0,0,0) (0.125,0.125,0) (0,0.25,0)
    };
    \addplot3[only marks,mark=*] coordinates {(0,0,0)} node[anchor=north east]{$x_1$};
    \addplot3[only marks,mark=*] coordinates {(0,0.25,0)} node[anchor=north]{$x_2$};

    \addplot3[thick,red] coordinates {
      (0,0.25,0) (0,0.3125,0.04625) (0,0.375,0)
    };
    \addplot3[only marks,mark=*] coordinates {(0,0.375,0)} node[anchor=north east]{$x_3$};

    \addplot3[thick,green!70!black] coordinates {
      (0,0.375,0) (-0.02,0.45,0.02) (0,0.4375,0)
    };
    \addplot3[only marks,mark=*] coordinates {(0,0.4375,0)} node[anchor=north]{$x_4$};
    \addplot3[only marks,mark=*] coordinates {(0,0.5,0)}
      node[anchor=north west]{$x_\infty$};

      \draw[->, dashed] (0,0,0) -- (-0.065,0,0.02) node[below right] {$e_4$};
   
  \end{axis}
\end{tikzpicture}

\end{figure}
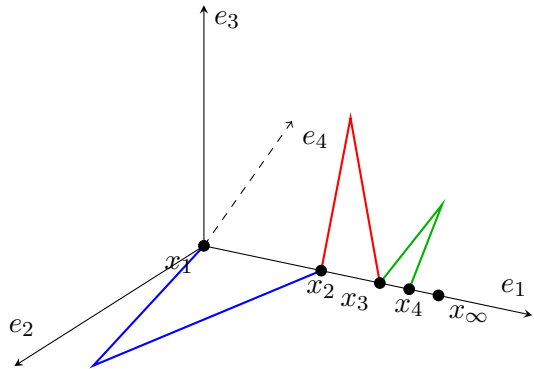

Endow $X$ with the metric $d$ inherited by $\ell^1(\N)$.
If follows that $X$ is a compact subset of $\ell^1(\N)$. Moreover, $X$ can be seen as the image of a rectifiable curve. Indeed, if $\gamma:[0,1]\to X$ is a continuous bijection, then
\(\mathrm{len}(\gamma):=\sum_{n=1}^\infty \dfrac{1}{2^{n+1}}2=1.\)
\hfill$\Diamond$
\end{example}
\begin{proposition}\label{prop:Pato-simple-properties}
    The space $(X,d)$ of Example~\ref{pato_simple} is a compact and $2$-quasiconvex metric space that fails to be eikonal.
\end{proposition}

\begin{proof}
     We start with quasiconvexity. Let $y,z\in X$. Assume that $y,z\in \frac{1}{2^{n+1}}S_n+x_n$ for some $n\in \N$. Then, it readily follows that $d_I(y,z)\leq 2 d(y,z)$. On the other hand, if $y\in S_m$ and $z\in S_n$, with $m<n$, we have
    \begin{align*}
      d_I(y,z)&=d_I(y, x_{m+1})+d_I(x_n,z)+\sum_{k=m+1}^{n-1}\underbrace{ d_I(x_k,x_{k+1})}_{2d(x_k,x_{k+1})} \\
      &\leq 2d(y, x_{m+1})+2d(x_n,z)+2d(x_{m+1},x_n)= 2d(y,z),
    \end{align*}
    where in the last equality we have used that the ambient space is $\ell^1(\N)$.
    \\
    
    We proceed to show that $X$ is not eikonal. Let $\Omega\subset X$ be any proper open set containing $x_{\infty}$, $\ell:\Omega\to(0,+\infty)$ continuous with $\inf \ell >0$, $g\in C(\partial\Omega)$ such that $(\ell,g)$ satisfy~\eqref{CC}, and consider the slope eikonal equation 
    \begin{equation}\label{eq.chopan}
        \begin{cases}
            s[u](x)=\ell(x),\quad\text{for all }x\in\Omega,\\
            u(x)=g(x), \quad\text{for all }x\in\partial\Omega.
        \end{cases}
    \end{equation}
    
    Assume that \eqref{eq.chopan} admits continuous solution $u$. Let $\gamma:[0,1]\to X$ be an arc-length parametrization (therefore $1$-Lipschitz) of $X$ such that $\gamma(0)=0$ and $\gamma(1)=x_\infty$. Then, thanks to Lemma~\ref{lax formula}, 
    we have that
    \[u(x_{\infty})-u(x_n)=\int_{t_{n}}^{1} \ell(\gamma(t))dt,
    \]
    where $\gamma(t_n)=x_n$. Choose $\varepsilon >0$. The continuity of $\ell\circ\gamma$ entails that $\ell(\gamma(t))\geq \ell(x_{\infty})-\varepsilon$ for all $t$ close enough to $1$.
    Since for every $n\in\N$ we have that $t_n:=1-\frac{1}{2^{n-1}}$, we estimate
    \[s[u](x_\infty)\geq \limsup_{n\to\infty}\dfrac{u(x_\infty)-u(x_n)}{d(x_\infty,x_n)}\geq(\ell(x_{\infty})-\varepsilon)\limsup_{n\to\infty}\frac{\displaystyle1-\left(1-\frac{1}{2^{n-1}}\right)}{\frac{1}{2^n}}=2(\ell(x_{\infty})-\varepsilon).\]
    We conclude that, $s[u](x_\infty)\geq 2\ell(x_{\infty})$ and so $u$ is not a pointwise solution of equation~\eqref{eq.chopan}.
\end{proof}
In summary, Example~\ref{circle}, Example~\ref{araña_no_QC} an Example~\ref{pato_simple} together imply that quasiconvexity (which is a bi-Lipschitz equivalence between $d$ and $d_I$) and the eikonal property are independent of each other.\\   

It is clear that the pathological behavior of Example~\ref{pato_simple} occurs solely around the point $x_\infty$. Our next example reproduces this behavior in a \emph{dense set}, so that no slope eikonal equation may admit a pointwise solution, while maintaining compactness and quasiconvexity of the space. Two technical lemmas are required for this construction. 

\begin{lemma}\label{pegado-quasiconvex}
Let $(X,d)$ be a $C$-quasiconvex complete metric space. For each $n\in\N$, let $r_n>0$ and $d_n$ be a metric on $[0,r_n]$, such that $([0,r_n],d_n)$ is $C$-quasiconvex. 
Let $Q:=\{q_n:~n\in\N\}\subseteq X$ be a countable set and define 
$$ Z=\left\{ (x,y)\in X\times [0,1]~:~\left\{\begin{array}{ll}
 y=0 &\text{ if } x\notin Q \\
  y\in [0,r_n] &\text{ if } x\in Q 
\end{array} \right.\right\},$$
\noindent endowed with the metric

\[\rho((x_1,y_1),(x_2,y_2))=\left\{\begin{array}{ll}
 d_n(y_1,y_2) &\text{ if } x_1=x_2=q_n,\\
 d_n(y_1,0)+d(x_1,x_2)+d_m(0,y_2) &\text{ if }x_1=q_n,~x_2=q_m,\text{ with }n\neq m,\\
 d_n(y_1,0)+d(x_1,x_2) &\text{ if } x_1=q_n~\text{ and }x_2\notin Q,\\
 d(x_1,x_2) &\text{ if }x_1,x_2\notin Q.
 
\end{array} \right.\]
Then, $(Z,\rho)$ is also $C-$quasiconvex and complete.
\end{lemma}
\begin{proof}
    Direct, for every pair of points $(x_1,y_1),(x_2,y_2)\in Z$ with $x_1\neq x_2$, we can consider the curve $\gamma$ that connects $(x_1,y_1)$ with $(x_1,0)$, then $(x_1,0)$ with $(x_2,0)$, and finally $(x_2,0)$ with $(x_2,y_2)$. 
\end{proof}
The following lemma is a straightforward consequence of the Arzelà-Ascoli theorem. For a metric space $(X,d)$ and two subset $A,B\subset X$, we denote the Hausdorff-Pompeiu distance from $A$ to $B$ by
\[d_H(A,B)=\max \left\{\sup_{x\in A}\,d(x,B),\sup_{y\in B}\,d(y,A)\right\},\]
where $d(\cdot,C):= \inf_{z\in C} d(\cdot,z)$ for any $C\subset X$.
\begin{lemma}\label{limit-of-quasiconvex}
Let $(X,d)$ be a compact metric space and let $\{X_k\}_k$ be a sequence of subsets of $X$. Denote by $d_k$ the restriction of $d$ to $X_k\times X_k$. Assume that there is a $C>0$ such that
\begin{itemize}
    \item $(X_k,d_k)$ is $C$-quasiconvex for every $k\in\N$, and
    \item $\lim_{k\to \infty}d_H(X_k,X)=0$,
\end{itemize}
Then, $(X,d)$ is also $C-$quasiconvex.
\end{lemma}
\begin{proof}
    Indeed, let $x,y\in X$ and $\{x_k\}_k,\{y_k\}_k\subset X$ be sequences, convergent to $x$ and $y$ respectively, such that $x_k,y_k\in X_k$ for all $k\in \N$.
    Since $(X_k,d_k)$ is a compact $C$-quasiconvex space, there is a $1$-Lipschitz curve $\gamma_k:[0,T_k]\to X_k$ such that
    \[d_{k,I}(x_k,y_k)=T_k=\mathrm{len}(\gamma_k)\leq C d_k(x_k,y_k),\]
    \noindent where $d_{k,I}$ is the intrinsic metric associated to $d_k$. Since $d_k=d$ on $X_k$, we have 
    \begin{equation}
    \label{eq: venegas}
    d_{I}(x_k,y_k)\leq T_k\leq C d(x_k,y_k).    
    \end{equation}
    Let $\{k(l)\}_l$ be a subsequence such that $\{T_{k(l)}\}_l$ converges to some $T\in\R$. Then, thanks to Arzelà-Ascoli theorem and up to a further subsequence, $\{\gamma_{k(l)}\}_l$  uniformly converges to a $1$-Lipschitz curve ${\gamma:[0,T]\to X}$ such that $\gamma(0)=x$ and $\gamma(T)=y$. Therefore, combining with~\eqref{eq: venegas} we get
    \[d_{I}(x,y)\leq \mathrm{len}(\gamma)\leq T\leq  \limsup_{l\to\infty} Cd (x_{k(l)},y_{k(l)})= C d (x,y).\]
    Since $x,y\in X$ are arbitrary, we get that $(X,d)$ is $C$-quasiconvex.
\end{proof}

\begin{example}\label{hyperpato}
Let us now consider $(X,\theta)$ be the metric space constructed in Example~\ref{pato_simple}. 
Let $\gamma:[0,1]\to X$ be the arc-length parametrization of $X$ such that $\gamma(0)=0$ and $\gamma(1)=x_\infty=\tfrac{1}{2}e_1$. For each $r>0$ we consider the distance $\theta_r$ on $[0,r]$ defined by
\[
\theta_r(s,t)= r\theta(\gamma(r^{-1}s),\gamma(r^{-1}t)),\quad \text{for all }s,t\in[0,r].
\]

Observe that $([0,r],\theta_r)$ is $2$-quasiconvex, that $\mathrm{diam}([0,r]) = r$, and that for all $n\in\N$ we have that
\begin{equation}\label{eq:Pato-effect}
\theta_{r,I}\left(r,r\left(1-\dfrac{1}{2^{n-1}}\right)\right)= r\theta_I(x_\infty,x_{n})=2r\theta(x_\infty,x_{n})=2\theta_r\left(r,r\left(1-\dfrac{1}{2^{n-1}}\right)\right).
\end{equation}
Let $d=\theta_1$ and denote $P_1:=[0,1]$ and $\rho_1:=d$. 
Set $L_1:=\{1\}$, which will represent the set of \textit{leaves} of $P_1$. 
Let $k\in\N$ and assume that $(P_k,\rho_k)$ and $L_k$ is already defined, and that $P_k\subset[0,1]^k$. 
Then, we define $(P_{k+1},\rho_{k+1})$ as the space constructed by applying Lemma~\ref{pegado-quasiconvex} with $(X,d)=(P_k,\rho_k)$,  $r_n = 2^{-(n+k)}$ and $d_n = \theta_{r_n}$, that is,
\[
d_n(s,t) := r_n d(r_n^{-1}s, r_n^{-1}t)=\dfrac{1}{2^{n+k}}d(2^{n+k}s,2^{n+k}t).
\]
We select $Q = Q_k\subset P_k\setminus L_k$ as a countable dense subset verifying that $\pi_k$ is one-to-one on $Q_k$ and $0\notin \pi_k(Q_k)$, where $\pi_k(x)$ stands for the $k$th coordinate of an element $x\in P_k\subset [0,1]^k$. Define now the set $L_{k+1}\subset P_{k+1}$ of leaves of $P_{k+1}$ by
\[L_{k+1}:=L_k\times \{0\}\cup \{(q_n,r_n)\,:\, q_n \in Q_k\},\]
and denote by $A_k:=\pi_k(Q_k)$ as the set of countable values for the $k$-th coordinate of the elements of $Q_k$.
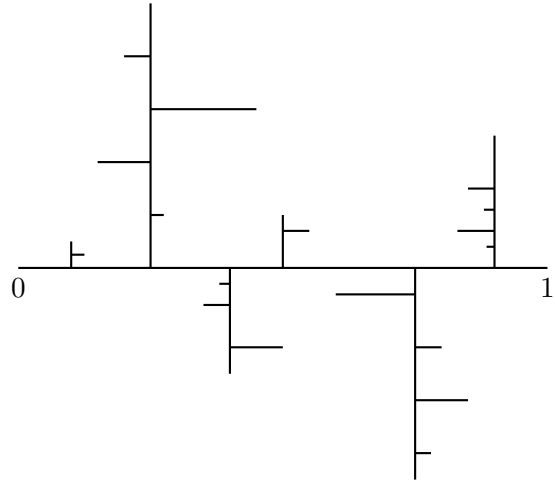
\begin{figure}[H]
\label{peine}
    \centering
\caption{Graphic representation of $P_3$. The end of lines are the leaves. The segments are drawn proportionally to their diameter. The construction only asks for density of new segments, without controlling their position. The diameter of new segments decays quickly to zero.
\vspace{1em}}
\begin{tikzpicture}[scale=7]

\draw[thick] (0,0) -- (1,0);
\node[below] at (0,0) {$0$};
\node[below] at (1,0) {$1$};

\draw[thick] (0.1,0)--(0.1,0.05);

\draw[thick] (0.25,0)--(0.25,0.5);
\draw[thick] (0.4,0)--(0.4,-0.2);

\draw[thick] (0.5,0)--(0.5,0.1);
\draw[thick] (0.75,0)--(0.75,-0.4);

\draw[thick] (0.9,0)--(0.9,0.25);

\draw[thick] (0.1,0.025)--(0.125,0.025);

\draw[thick] (0.25,0.4)--(0.2,0.4);
\draw[thick] (0.25,0.3)--(0.45,0.3);
\draw[thick] (0.25,0.1)--(0.275,0.1);
\draw[thick] (0.25,0.2)--(0.15,0.2);

\draw[thick] (0.4,-0.15)--(0.5,-0.15);
\draw[thick] (0.35,-0.07)--(0.4,-0.07);
\draw[thick] (0.38,-0.03)--(0.4,-0.03);

\draw[thick] (0.5,0.07)--(0.55,0.07);

\draw[thick] (0.75,-0.35)--(0.78,-0.35);
\draw[thick] (0.75,-0.25)--(0.85,-0.25);
\draw[thick] (0.75,-0.15)--(0.8,-0.15);
\draw[thick] (0.6,-0.05)--(0.75,-0.05);

\draw[thick] (0.9,0.15)--(0.85,0.15);
\draw[thick] (0.9,0.11)--(0.88,0.11);
\draw[thick] (0.9,0.07)--(0.83,0.07);
\draw[thick] (0.9,0.04)--(0.885,0.04);
\end{tikzpicture}

\end{figure}

Define, for $k_1<k_2$, the injection $i_{k_1,k_2}:P_{k_1}\to P_{k_2}$ by
\[i_{k_1,k_2}(x):=x\times\{0\}^{k_2-k_1},\quad \text{for all }x\in P_{k_1}.\]
It readily follows that $i_{k_1,k_2}$ is an isometry.
Also, for each $k\in\N$, define $i_k:P_k\to [0,1]^\N$ by
\[i_k(x):= x\times\{0\}^{\N}=(x_1,...,x_k,0,...),\quad \text{for all }x\in P_k.\]
Note that, for any $k_1<k_2$, we have $i_{k_1}= i_{k_2}\circ i_{k_1,k_2}$.
Also, define $\mathcal{P}_\infty\subset[0,1]^{\N}$ by 
\[\mathcal{P}_\infty:=\bigcup_{k\in \N}i_k(P_k)\]
and the metric $\rho_\infty $ on $\mathcal{P}_\infty\subset[0,1]^{\N}$ by
\[
\rho_\infty(x,y):=\rho_k(i_k^{-1}(x),i_k^{-1}(y)),\quad \text{where }k\in\N~~ \text{and }~x,y\in i_k(P_k),
\]
which is well defined (and a metric) due to the fact that the maps $\{i_{k_1,k_2}\}_{k_1<k_2}$ are isometries. Note that the union defining $\mathcal{P}_\infty$ is over an increasing family of sets. Finally, denote by $(P_\infty,\rho_\infty)$ the completion of $(\mathcal{P}_\infty,\rho_\infty)$ and extend $\{i_k\}$ to have  range in $P_\infty$. \hfill$\Diamond$
\end{example}

Note that this construction has the following invariant: For each $k$, and each ${x\in \mathcal{P}_{\infty}\setminus i_{k}(P_{k})}$, one has that $\pi_{k}(x)$ belongs to the countable set of possible values, given by $A_{k}$. Here $\pi_{k}:\mathcal{P}_{\infty}\to [0,1]$ is the natural extension of $\pi_{k}$ that assigns the $k$-th coordinate to elements of $\mathcal{P}_{\infty}\subset [0,1]^{\N}$. This invariant can be extended to the whole space $(P_{\infty},\rho_{\infty}$) as shown in the following lemma.

\begin{lemma}\label{Lemma:TheOneshouldNotBeNamed}
    The function $\pi_k:\mathcal{P}_\infty\to [0,1]$ is continuous and it can be continuously extended to $P_\infty$. 
    Moreover, for every $\overline{x}\in P_\infty\setminus i_k(P_k)$ we have that $\pi_k(\overline{x})\in A_k:= \pi_k(Q_k)$.
\end{lemma}

\begin{proof}
    Note that by construction $\pi_k$ is uniformly continuous when restricted to $i_k(P_k)$. Observe also that for every $x\in \mathcal{P}_{\infty}\setminus i_k(P_k)$, then $\pi_k(x) = \pi_k(q)$, where $q$ is the nearest point of $i_k(P_k)$ to $x$, which coincides with the unique element $q\in Q_k$ from which the branch containing $x$ emanates in the iterative construction. Then, denoting $\proj_k:\mathcal{P}_{\infty} \to i_k(P_k)$ as the metric projection onto $i_k(P_k)$, we deduce that $\pi_k = \pi_k \circ \proj_k$. Noting that for every $x,y\in \mathcal{P}_{\infty}$, we have that 
    \[
    \rho_{\infty}(x,y) \geq \rho_{k}(\proj_k(x),\proj_k(y)),
    \]
    we conclude that $\pi_k$ remains uniformly continuous on $\mathcal{P}_{\infty}$. Thus, it admits a unique uniformly continuous extension to $P_{\infty}$. It remains to show that  $\pi_k(x) \in A_k$ for every $x\in P_{\infty}\setminus\mathcal{P}_{\infty}$.\\

    Let $(x_n)\in \mathcal{P}_{\infty}$ such that $x_n\to x$. For $n\in\N$ large enough, we have that $x_n\notin i_{k}(P_k)$ and so $\pi_k(x_n) \in A_k$. We show that $\pi_k(x_n)$ is eventually constant. If not, we have that $\{\pi_k(x_n)\}$ needs to be infinite, and so it must be the set $\{\proj_k(x_n)\}$. Thus, up to a subsequence, we can assume that $n\in \N\mapsto q_n=\proj_k(x_n)\in Q_k$ is one-to-one. Therefore, each element of $(x_n)$ belongs to a different branch of $\mathcal{P}_{\infty}\setminus i_k(P_k)$, the one emanating from $q_n$. By construction, we have that
    \[
    \mathrm{diam}(\proj_k^{-1}(q_n))\xrightarrow{n\to\infty} 0.
    \]
Indeed, note that the sets $\{\proj_k^{-1}(q_n)\,:\,n\in\N\}$ are pairwise disjoint. Thus, the sum of all terms $\mathrm{diam}(\proj_k^{-1}(q_n))$ is less than the sum of all diameters of all segments glued after the iteration $k$ in the construction of Example~\ref{hyperpato}. Furthermore, for each $r>k$, the sum of the diameters of all segments glued at iteration $j$ is given by $\sum_{n=1}^{\infty}\frac{1}{2^{n+j}} = \frac{1}{2^j}$. Thus,
\[
\sum_{n=1}^{\infty}\mathrm{diam}(\proj_k^{-1}(q_n)) \leq \sum_{j=k+1}^{\infty} \frac{1}{2^j} = \frac{1}{2^k},
\]
and so, the summand converges to zero as we claimed. We deduce that 
\[
\rho_{\infty}(x,i_k(P_k)) = \lim_n \rho_{\infty}(x_n,i_k(P_k)) \leq \lim_n\mathrm{diam}(\proj_k^{-1}(q_n)) = 0,
\]
a contradiction. Thus, $\pi_k(x_n)$ must be eventually constant and so $\pi_k(x)\in A_k$. The proof is complete.
\end{proof}

\begin{proposition}
    For every $k\in\N$, $(P_k,\rho_k)$ is a compact $2$-quasiconvex metric space. Moreover, the sequence of compact subspaces $i_k(P_k)$ converges to the compact $2$-quasiconvex space $(P_\infty,\rho_\infty)$ in the sense of the Hausdorff-Pompeiu distance. 
\end{proposition}
\begin{proof}
   Let us start by proving compactness and quasiconvexity of $(P_k,\rho_k)$. We know that $(P_1,\rho_1)$ is a compact $2$-quasiconvex space. We proceed by induction: let $k\geq 2$ and assume that $(P_{k-1},\rho_{k-1})$ is compact and $2$ quasiconvex. Thanks to Lemma~\ref{pegado-quasiconvex}, we have that $(P_k,\rho_k)$ is complete and $2$-quasiconvex as well. Let us check that is compact by showing that it is totally bounded. Indeed, let $\varepsilon>0$. Since $P_{k-1}$ is compact, there is a finite set $\{x_n\}_{n=1}^N\subset P_{k-1}$ such that
    \[P_{k-1}\subset\bigcup_{n=1}^N B\left(x_n,\frac{\varepsilon}{2}\right).\]
    Define now $y_n:=i_{k,k+1}(x_n)\in P_k$ and note that 
    \[\bigcup_{n=1}^N B(y_n,\varepsilon)\supset O_k:= \bigcup_{x\in P_{k-1}} B\left(i_{k,k+1}(x),\frac{\varepsilon}{2}\right).\]
    Also, by construction of $P_k$, the set $P_k\setminus O_k$ consists only on finitely many bounded segments. This readily yields that $P_k$ is totally bounded. 
    \\
    
    Next, we check convergence of the sequence $i_k(P_k)$ to $P_\infty$. Let $\rho_{k,H}$ be the Hausdorff distance with respect to the metric $\rho_k$ on $P_k$ for $k\in\N\cup\{\infty\}$. By construction of $P_k$, for any $k\in\N$ 
    \[\rho_{k+1,H}(i_{k,k+1}(P_k),P_{k+1})\leq \frac{1}{2^k}.\] 
    Therefore, it follows that
    \[\rho_{\infty,H}(i_{k}(P_k),P_{\infty})= \rho_{\infty,H}(i_{k}(P_k),\mathcal{P}_{\infty})\leq \sum_{n\geq k}^\infty \frac{1}{2^n}=\dfrac{1}{2^{k-1}}.\] 
    So, $i_k(P_k)$ converges to $P_\infty$ with respect to the Hausdorff-Pompeiu distance. Hence, $P_\infty$ is a compact space. The $2$-quasiconvexity of $(P_\infty,\rho_\infty)$ follows from Lemma~\ref{limit-of-quasiconvex}.
\end{proof}
Let us define the subset $L_\infty$ of \textit{leaves} of $P_\infty$ by
\[L_\infty := \bigcup_{k=1}^\infty i_k(L_k).\]
\begin{proposition}\label{prop: density}
    $L_\infty$ is dense subset of $(P_\infty,\rho_\infty)$.
\end{proposition}
\begin{proof}
Fix $k\in \N$. By construction of $P_k$ and $L_k$, we have $\rho_{k,H}(L_k,P_k)\leq 2^{-k}$. Also, we have that $\rho_{\infty,H}(i(P_k),P_\infty)=\rho_{\infty,H}(i(P_k),\mathcal{P}_\infty)\leq 2^{-k+1}.$
Thus, $\rho_{\infty,H}(i_k(L_k),P_\infty)\leq 2^{-k+1}+2^{-k}$.
Hence, $L_\infty$ is dense in $P_\infty$.
\end{proof}

\begin{proposition}\label{hyperpato_antiEikonal}
    There is no solvable slope eikonal equation of the form~\eqref{eq: eikonal complete} on the space $(P_\infty, \rho_\infty)$.
\end{proposition}
\begin{proof}
    Let $\Omega\subsetneq P_\infty$ be a nonempty open set, $\ell\in \mathcal{C}_b(\Omega)$, with $\inf_\Omega \ell>0$, and $g\in C(\partial \Omega)$ such that the pair $(\ell,g)$ satisfies~\eqref{CC}.
    Assume that the equation
    \[\begin{cases}s[u](x)=\ell(x),\quad x\in \Omega\\
    u(x)=g(x),\quad x\in \partial \Omega\end{cases}\]
    has a continuous pointwise solution $u:\overline{\Omega}\to\R$. 
    Due to the density of $L_\infty$, there exists $\overline x\in L_\infty\cap \Omega$. 
    Thanks to Lemma~\ref{lax formula}, there is a $1$-Lipschitz curve $\gamma:[0,T]\to \overline{\Omega}$ such that $\gamma(0)=\overline x$, $\gamma(T)\in \partial \Omega$, $\gamma([0,T))\subset \Omega$ and
    \[u(\overline x)=\int_0^T\ell(\gamma(t))dt + g(\gamma(T)).\]
    Since $\overline x\in L_\infty$, there is a (unique) $k\in \N$ such that $\overline{x}\in i_k(P_k)\setminus i_{k-1}(P_{k-1})$. 
    We claim that there is $\varepsilon>0$ such that $\gamma([0,\varepsilon))\subset i_k(P_k)$. Since $\gamma$ and $\pi_k$ are continuous, then $\pi_k\circ\gamma$ is continuous as well. Note also that there is $\delta>0$ such that 
    \[\pi_k(x)\geq\pi_k(\overline{x})~\text{and}~x\neq \overline{x} \implies \rho_\infty (x,\overline{x})\geq \delta.\]
    Indeed, $\delta$ can be chosen as  $\rho_{\infty}(\overline{x},i_{k-1}(P_{k-1}))$.
    Let $r\in (0,\delta)$ small enough such that $\pi_{1},\ldots, \pi_{k-1}$ are constant on $i_k(P_k)\cap B(\bar{x},r)$, and choose $t>0$ small enough such that $\gamma([0,t])\subset B(\bar{x},r)$. Injectivity of $\gamma$ implies that $\pi_k(\gamma(t))<\pi_k(\overline x)$. Thus, $\pi_k\circ\gamma([0,t])\supset [\pi_k(\gamma(t)),\pi_k(\overline x)]$. 
    Let $\varepsilon\in (0,t)$ be such that 
    \[\pi_k(\gamma(\varepsilon))\notin A_k\quad\text{and}\quad \rho_\infty (\overline{x},\gamma(\varepsilon))\leq \rho_{\infty}(\overline x, i_{k-1}(P_{k-1})).\]
    Thus, by definition of $A_k$, we have $\gamma(\varepsilon)\in i_k(P_k)$.
    Now we prove that $\gamma((0,\varepsilon))\subset i_k(P_k)$. 
    Indeed, otherwise there is $\overline{s}\in (0,\varepsilon)$ such that $\gamma(\overline s)\notin i_k(P_k)$.
    Since $\gamma$ is injective, it follows that ${\pi_k(\gamma(\overline s))\in (\pi_k(\gamma(\varepsilon)), \pi_k(\bar{x}))}$.
    Let us define
            \[s_1:=\sup\{s<\overline s:~\gamma(s)\in i_k(P_k)\}\quad\text{and}\quad s_2:=\inf\{s> \overline{s}:~ \gamma(s)\in i_k(P_k)\}.\]
    By the choice of $\gamma(0)$ and $\gamma(\varepsilon)$, $s_1$ and $s_2$ are well defined. Moreover, thanks to the continuity of $\gamma$ and the compactness of $i_k(P_k)$, we also know that $s_1<\overline{s}<s_2$ and $\gamma(s_1),\gamma(s_2)\in i_k(P_k)$. Also, we have that
    \[
    \pi_k\circ\gamma ((s_1,s_2))\subset A_k.
    \]
    Indeed, $\gamma ((s_1,s_2))\subset P_{\infty}\setminus i_k(P_k)$, see Lemma~\ref{Lemma:TheOneshouldNotBeNamed}.  Since $A_k$ is countable, necessarily $\pi_k\circ\gamma ((s_1,s_2))$ is a singleton, and so $\pi_k(\gamma(s_1))=\pi_k(\gamma(s_2))$. Since $\pi_j(\gamma(s_1)) = \pi_j(\gamma(s_2))$ for all $j\in\{1,\ldots,k-1\}$ and $\gamma(s_1),\gamma(s_2)\in i_k(P_k)$,  
    we deduce that
    $\gamma(s_1)=\gamma(s_2)$. This contradicts the injectivity of $\gamma$.\\
    
    We proceed to prove that $s[u](\bar{x})>\ell(\bar{x})$. By construction, on $P_{k}$ we have that $\bar{x} = (q_m,r_m)$ for some $q_m\in Q_k$, with $r_m = 2^{-(k+m)}$. Moreover, by~\eqref{eq:Pato-effect}, we have that
    \[\rho_{k,I}\left((q_m,r_m),\left(q_m, r_m-\frac{r_m}{2^n}\right)\right) = 2\rho_k\left((q_m,r_m),\left(q, r_m-\frac{r_m}{2^n}\right)\right)\quad\text{for all }n\in \N.\]
    Combining this with the fact that $i_k:P_k\to P_\infty$ is an isometry, we can continue as in the proof of Example~\ref{pato_simple}. Using continuity of $\ell\circ \gamma$, we conclude that
    $s[u](\bar{x})\geq 2\ell(\bar{x})$.
\end{proof}
We have shown that the construction of Example~\ref{hyperpato} is a compact a quasiconvex metric space where no eikonal equation admits solution. This can be regarded as an extreme failure of the stability of the eikonal property for metric spaces under bi-Lipschitz homeomorphism (in this case, $\mathrm{Id}:(P_\infty,\rho_\infty)\to (P_\infty,\rho_{\infty,I})$). 

\section{Proof of Theorem~\ref{thm: main}}\label{sec:main}
The proof of Theorem~\ref{thm: main} is rather long and it will require several intermediate steps. Let us briefly summarize them.
We first define the class of spaces  that satisfy the hypothesis of Theorem~\ref{thm: main} (the \textit{$E_{1,0}$-property} below) to deduce that, in these spaces, the intrinsic distance is well defined and continuous. 
Then, we proceed with the continuity of the natural candidate of solution of any eikonal equation (i.e. the one provided by Corollary~\ref{cor: uniqueness}).
In the next step we show that in compact spaces with the property $E_{1,0}$ we can solve eikonal equations with more general boundary data $g\in C(\partial \Omega)$. 
In a very unexpected but useful turn, we use the previous step to prove the following key observation: If $(X,d)$ has the $E_{1,0}$-property, then the product space $(X\times [0,1],\rho)$, equipped with the distance $\rho:=d+|\cdot|$, also enjoys the $E_{1,0}$-property. We end this section with the last arguments to prove Theorem~\ref{thm: main}.\\

Let us start with the following definitions.
\begin{definition}
\phantom{ii}
\begin{enumerate}
    \item A space $(X,d)$ has the $E_{1,0}$-property if every eikonal equation over an arbitrary nonempty open set $\Omega\subsetneq X$, with data $\ell\equiv 1$ and $g\equiv 0$, admits a continuous pointwise solution. (i.e. the hypothesis of Theorem~\ref{thm: main}.)
    \item A space $(X,d)$ has the $E_{1,g}$-property if every eikonal equation over an arbitrary nonempty open set $\Omega\subsetneq X$, with data $\ell\equiv 1$ and $(\ell,g)$ satisfying the compatibility condition \eqref{CC} admits a continuous pointwise solution. 
\end{enumerate}
\end{definition}

\begin{proposition}
    Assume that $(X,d)$ is a compact metric space that has the $E_{1,0}$-property. Then the intrinsic distance ${d_I:X\times X\to \R}$ is well defined and continuous. 
\end{proposition}
\begin{proof}
    As a direct consequence of having the $E_{1,0}$-property and Lemma~\ref{lax formula}, we deduce that for any nonempty closed set $F\subset X$, the intrinsic distance function to $F$, $d_I(\cdot,F)$, is the unique solution of the equation
    \[\begin{cases}
        s[u]=1,\quad&x\in X\setminus F\\
        \phantom{iii}u=0,& x\in F.
    \end{cases}\]
    Therefore, $d_I(\cdot,F)$ is well defined and continuous.
    Now, let $\{x_n\}_n,\{y_n\}_n\subset X$ and $\overline{x},\overline{y}\in X$ be such that $\lim_nx_n=\overline x$ and $\lim_n y_n=\overline{x}$. 
    It easily follows that
    \begin{align*}
        |d_I(x_n,y_n)-d_I(\overline{x},\overline{y})|\leq d_I(x_n,\overline{x}) +d_I(y_n,\overline{y}),\quad\text{for all }n\in \N.
    \end{align*}
    Thus, we get $\lim_n d_I(x_n,y_n)=d_I(\overline{x},\overline{y})$. The continuity of $d_I$ follows.
\end{proof}
Now we prove the continuity of our candidate for solution.
\begin{lemma}\label{lemma:ContinuityV}
  Let $(X,d)$ be a compact metric space that has the $E_{1,0}$-property. Let $\Omega \subsetneq X$ be a nonempty open set, $\ell\in \mathcal{C}_{b}(\Omega)$ with $\inf_{\Omega}\ell >0$, and $g\in C(\partial\Omega)$  such that the pair $(\ell,g)$ satisfies~\eqref{CC}. Then the function $V:\overline{\Omega}\to\R$ defined by
  \[
    V(x):=\inf\left \{\int_0^T \ell(\gamma(t))dt+g(\gamma(T))\,\colon \begin{array}{l}
     ~ T>0,~\gamma\in \mathrm{Lip}_1([0,T],\overline{\Omega}),\\
     \gamma(0)=x,~\gamma(T)\in \partial \Omega,~\gamma((0,T))\subset\Omega
     \end{array}\right\}.
     \]
     is continuous.
\end{lemma}
\begin{proof}
    Note that the compatibility condition~\eqref{CC} entails that $V = g$ on $\partial \Omega$. 
    Moreover, since $d_{I}$ is continuous, we deduce that $V$ is continuous on $\Omega$. 
    Indeed, take $x\in \Omega$ and choose $\delta>0$ small enough such that $B_{I}(x,\delta)\subset \Omega$. 
    Then, using the dynamic programming principle, we deduce that
\[
|V(x) - V(y)| \leq Ld_I(y,x),\quad \forall y\in B_I(x,\delta),
\]
where $L=\sup_{\Omega}\ell$. In particular, continuity of $d_I$ yields the continuity of $V$ at $x$.\\

Now, fix $x\in\partial \Omega$. Let $y\in \overline{\Omega}$ and let $\gamma_y:[0,T]\to X$ be a $1$-Lipschitz curve connecting $y$ and $x$ such that $\mathrm{len}(\gamma_{y}) = d_I(x,y)$. 
Set $t_y = \sup\{ s\geq 0\mid \gamma_{y}((0,s))\subset \Omega\}$ (or $t_y=0$ if there is no $s>0$ such that $(0,s)\subset \Omega$) and $z_y = \gamma_{y}(t_y)$. 
It follows that $z_y\in\partial\Omega$. 
Since $d_I(\cdot,x)$ is continuous, it follows that $z_y$ tends to $x$ as $y$ tends to $x$. Then, we compute
\begin{align*}
    \liminf_{y\in \overline{\Omega}\to x} V(y)\leq 
    \liminf_{y\in \overline{\Omega}\to x} Ld_I(y,x)+g(z_y)=g(x)= V(x).
\end{align*}

where the second last equality follows from the continuity of $g$. Now, suppose that there exist a sequence $\{y_n\}_n\subset \overline{\Omega}$, convergent to $x$, and $\sigma>0$ such that $\lim_{n\to\infty} V(y_n) = V(x) - \sigma$. 
Since $g\in C(\partial \Omega)$, we have that $\{y_n\}_n\subset \Omega$.
Thanks to the Arzela-Ascoli theorem, for all $n\in \N$, the infimum defining $V(y_n)$ is in fact a minimum.
Therefore, we can select a $1$-Lipschitz curve $\nu_{n}:[0,S_n]\to \overline{\Omega}$ and a point $w_n\in \partial \Omega$ such that $\nu_n((0,S_n))\subset \Omega$, $\nu_n(0) = y_n$, $\nu_n(S_n) = w_n$ and
\[
V(y_n) = \int_0^{S_n}\ell(\nu_n(t))dt + g(w_n).
\]
Let $\eta_n:[-S_n,t_{y_n}] \to \overline{\Omega}$ be the concatenation of $\widehat{\nu}_n(\cdot):[-S_n,0]\to\overline{\Omega}$ and $\gamma_{y_n}|_{[0,t_{y_n}]}$ connecting $w_n$ and $z_{y_n}$, where $\widehat{\nu}_n(t)=\nu_n(-t)$. Then, fix $n\in\N$ large enough so that $V(y_n)\leq V(x)-\sigma/3$, $d_I(z_{y_n},y_n)\leq \sigma/6L$ and $g(z_{y_n})\geq V(x) - \sigma/2$. We can compute
\begin{align*}
    g(w_n) + \int_{-S_n}^{t_{y_n}}\ell(\eta_n(t))dt &\leq V(y_n)+\int_0^{t_{y_n}}\ell(\gamma_{y_n}(t))dt\\
    &\leq V(y_n) + Ld_I(z_{y_n},y_n) \\
    &\leq  V(x) - \frac{\sigma}{2} \leq g(z_{y_n}).
\end{align*}
This is a contradiction with the compatibility condition~\eqref{CC} between the points $w_n$ and $z_{y_n}$.
Thus, $V$ must be continuous at $x\in\partial \Omega$.
\end{proof}

The following proposition shows that, for compact spaces, the $E_{1,0}$-property is equivalent to the $E_{1,g}$-property.
\begin{proposition}\label{prop:E01->E0g}
    Let $(X,d)$ be a compact metric space that has the $E_{1,0}$-property. Then,$(X,d)$ enjoys the $E_{1,g}$-property.
\end{proposition}
\begin{proof}
Recall that the $E_{1,0}$-property in compact spaces implies that the intrinsic distance to any nonempty closed subset is continuous.\\

Set $\Omega\subsetneq X$ a nonempty open set and $g\in C(\partial\Omega)$ such that $(1,g)$ verifies the compatibility condition \eqref{CC} (i.e. $\ell\equiv 1$ in $\Omega$). We need to show that
    \begin{equation}\label{eq:InProof-eikonal-(1,g)}
        \begin{cases}
            s[u](x)=1,\quad&x\in\Omega,\\
            u(x)=g(x),& x\in \partial\Omega, 
        \end{cases}
    \end{equation}
    admits a continuous solution. Thanks to Corollary \ref{cor: uniqueness}, the only candidate of solution $V:\overline{\Omega}\to \R$ of the above equation is given by
    \[
    V(x):=\inf\left \{\int_0^T 1dt+g(\gamma(T))\,\colon \begin{array}{l}
     ~ T>0,~\gamma\in \mathrm{Lip}_1([0,T],\overline{\Omega}),\\
     \gamma(0)=x,~\gamma(T)\in \partial \Omega,~\gamma([0,T))\subset\Omega
     \end{array}\right\},
     \]

which is continuous by Lemma \ref{lemma:ContinuityV}. Now, choose $x\in \Omega$. Our aim is to show that $s[V](x) = 1$. To do so, fix $\varepsilon>0$ and consider the set $K:= \{y\in \Omega\mid V(y)\geq V(x) - \varepsilon\}$. Define $\Omega_{\varepsilon} = \mathrm{int}(K)$. The continuity of $V$ ensures that $\mathrm{int}(K)\neq\emptyset$ and that for every $z\in \partial K \setminus\partial\Omega$, one has that $V(z) = V(x)-\varepsilon$. Then, we can consider the auxiliary problem
\begin{equation*}
        \begin{cases}
            s[W](y)=1,\quad&x\in\Omega_{\varepsilon},\\
            W(y)=V(x) - \varepsilon,& y\in \partial\Omega_{\varepsilon}, 
        \end{cases}
    \end{equation*}

Let $\delta = d_I(x,\partial \Omega)>0$. By shrinking $\varepsilon$ if necessary, we can choose $\rho>0$  small enough such that for every $y\in B(x,\rho)$, one has that $d_I(y,\partial K)<\frac{\delta}{2}< d_I(y,\partial \Omega)$. Let $\gamma_{y}:[0,T_y]\to \overline{\Omega}$ be the curve such that $V(y) = T_y + g(\gamma_y(T_y))$ and let $t_y = \sup\{s\geq 0\mid \gamma_y([0,s])\subset K\}$. Note that $\gamma_y(t_y)\in\partial K\setminus\partial\Omega$ and so $V(\gamma_y(t_y)) = V(x)-\varepsilon$.  Using Lemma~\ref{lax formula} and the dynamic programming principle, we get that
\[
V(y) = t_y + V(\gamma_{y}(t_y)) = t_y + V(x)- \varepsilon = W(y).
\]
and so $V$ and $W$ coincide on $B(x,\rho)$, which yields that $s[V](x) = s[W](x)$. Since $X$ has the $E_{1,0}$-property, we deduce that $s[W](x)=1$. The proof is now complete.
\end{proof}
Now we deal with the product space $X\times[0,1]$.
\begin{proposition}\label{proposition:E0g->E01 prod}
    Let $(X,d)$ be a compact metric space. Consider $\widetilde{X}:=X\times [0,1]$ endowed with the 1-distance $\rho$, that is,
    \[
    \rho((x,t),(y,s)) = d(x,y) + |t-s|.
    \]
    Suppose that $(X,d)$ has the $E_{1,0}$ property. Then, $(\widetilde{X},\rho)$ has the $E_{1,0}$ property.
\end{proposition}
\begin{proof}
    Let $\Omega\subsetneq \widetilde{X}$ be a nonempty open set. 
    We show that the equation
    \[\begin{cases}
        s[u]=1,\quad \text{in }\Omega\\
        u=0,\quad \text{in }\partial \Omega,
    \end{cases}\]
    admits a continuous solution.
    Observe that the intrinsic distance $\rho_I:\widetilde{X}\times\widetilde{X}\to \R$ is well defined and satisfies
    \begin{align}\label{well defined}
    \rho_I((x,s),(y,t))= d_I(x,y)+|s-t|,\quad\text{for all }(x,t),(y,s)\in \widetilde{X}.    
    \end{align}
    Therefore, the function $V:\overline{\Omega}\to \R$ defined by $V(x,t)=\rho_I((x,t),\partial \Omega)$ is well defined and continuous. 
    Let us compute the slope of $V$ in $\Omega$. 
    Let $\overline{z}:=(\overline{x},\overline t)\in \Omega$. Since $\widetilde{X}$ is a compact space and the intrinsic distance is finite valued, there is a $1$-Lipschitz curve $\gamma:[0,\rho_I(\overline{z},\partial \Omega)]\to \overline{\Omega}$ such that $\gamma(0)=\overline{z}$ and $\gamma(\rho_I(\overline{z},\partial \Omega))\in \partial \Omega$.
    Observe now that
    \begin{align*}
        s[V](\overline{z})\geq \lim_{t\to 0}\dfrac{\rho_I(\overline{z},\partial \Omega)-\rho_I(\gamma(t),\partial \Omega)}{\rho(\overline{z},\gamma(t))}\geq\lim_{t\to 0}\dfrac{\rho_I(\overline{z},\partial \Omega)-(\rho_I(\overline{z},\partial \Omega)-t)}{t}=1
    \end{align*}
    In what follows, we show that  $s[V](\overline{z})\leq 1$. Let $\{z_k\}_k:=\{(x_k,t_k)\}_k\subset \Omega $ be such that $z_k\to \overline{z}$ and such that
    \[
    \limsup_{k\to \infty} \frac{V(\bar{z}) - V(z_k)}{\rho(\bar{z},z_k)} = s[V](\bar{z}).
    \]
    
    By compactness of $\overline{\Omega}$, for each $k\in \N$, there is $w_k=(y_k,s_k)\in \partial \Omega$ such that
    \[V(z_k)=\rho_I(z_k,\partial \Omega)=\rho_I(z_k,w_k),\quad \text{for all }k\in \N.\]
    By compactness of $\partial \Omega$ and 
    up to a subsequence of $\{z_k\}_k$, we can and shall assume that\\ ${\lim_{k\to\infty} w_k=w_\infty=(y_\infty,s_\infty)}$ exists. Therefore,
    \[\rho_I(\overline{z},\partial \Omega)=\rho_I(\overline{z},w_{\infty}).\]
    We know show that $s[V](\bar{z})\leq 1$ by showing that the values of $\rho_I(z_k,w_k)$ and $\rho_I(\bar{z},w_{\infty})$ can be produced as a solution of a compatible pair $(1,g)$ in $X$. 
    We split the analysis in two cases.\\

    Assume first that $t_k=\overline t$ for all $k$. In this case, set $S:=\{y_k:~k\in \N\cup\{\infty\}\}$ and $g:S\to \R$ by
    \[g(y_k)=|s_k-\overline{t}|,\quad\text{for all }k\in \N\cup\{\infty\}.\]
    We check that the pair $(1,g)$ satisfies the compatibility condition $(CC)$ on $X$. Indeed, notice first that for any $k,j\in \N\cup\{\infty\}$
    \begin{align*}
        d_I(x_k,y_k)+|s_k-\bar{t}| &= \rho_I(z_k,w_k)=\rho_I(z_k,\partial \Omega)\leq \rho_I(z_k,w_j)\\
        & \leq d_I(x_k,y_j)+|s_j-\bar{t}|\leq d_I(x_k,y_k)+d_I(y_k,y_j)+|s_j-\bar{t}|.
    \end{align*}
    So, we have
    \[g(y_k)=|s_k-\overline t|\leq d_I(y_k,y_j)+|s_j-\overline{t}|= d_I(y_k,y_j)+g(y_j).\]
    Therefore, since $X$ is has the $E_{1,g}$-property by Proposition~\ref{prop:E01->E0g}, the equation
    \[\begin{cases}
        s[u]=1,\quad &\text{in }X\setminus S,\\
        u(y_k)=g(y_k),\quad &\text{for all }k\in \N\cup\{\infty\},
    \end{cases}\]
    admits a continuous solution $W:X\to\R $. 
    Observe that 
    \[W(x_k)= d_I(x_k,y_k)+g(y_k)= \rho_I(z_k,w_k)=V(x_k),\quad\text{for all }k\in\N.\]
    Thus, $W(\overline{x})=V(\overline{z})$.
    Hence,
    \begin{align*}
       s[V](\bar{z}) = \limsup_{k\to\infty} \dfrac{V(\overline{z})-V(z_k)}{\rho(\overline{z},z_k)}=\limsup_{k\to\infty} \dfrac{W(\overline{x})-W(x_k)}{d(\overline{x},x_k)}\leq S[W](\overline{x})=1. 
    \end{align*}

    Now, for the general case, note that for all $k\in \N$ we have 
    \[V(z_k)=\rho_I(z_k,\partial \Omega)\geq \rho_I((x_k,\overline t),\partial \Omega)-|t_k-\overline{t}|.\]
    Let $\varepsilon>0$. Thanks to the first case, we have that
    \[\limsup_{k\to\infty}\frac{V(\overline{z})-V(x_k,\overline t)}{d(x_k,\overline{x})}\leq 1.\]
    Therefore, there is $k_0\in \N$ such that
    \[V(\overline{z})-V(x_k,\overline t)\leq (1+\varepsilon)d(x_k,\overline{x}),\quad \text{for all }k\geq k_0.\]
    So, for any $k\geq k_0$ we get
    \begin{align*}
       s[V](\bar{z}) = \dfrac{V(\overline{z})-V(z_k)}{\rho(\overline{z},z_k)}\leq \dfrac{\rho_I(\overline{z},\partial \Omega)- \rho_I((x_k,\bar{t}),\partial \Omega)+|t_k-\overline{t}|)}{d(x_k,\overline{x})+|t_k-\overline t|} \\
        \leq \dfrac{(1+\varepsilon)d(x_k,\overline{x})+|t_k-\overline t|}{d(x_k,\overline{x})+|t_k-\overline t|}\leq 1+\varepsilon.
    \end{align*}

    By taking  infimum on $\varepsilon>0$, we get that $S[V](\overline{z})\leq 1$.\\
   
   Since $\bar{z}\in\Omega$ is arbitrary, we deduce that $V:\overline{\Omega}\to\R$ is a continuous function such that $S[V]=1$ on $\Omega$ and $V(\partial \Omega)=\{0\}$.
    Since $\Omega$ is an arbitrary nontrivial open subset of $X$, we finally deduce that $(\widetilde{X},\rho)$ has the $E_{1,0}$-property.
\end{proof}

Finally we can prove Theorem~\ref{thm: main}.
\begin{proof}[Proof of Theorem~\ref{thm: main}]
    We denote by $\widetilde{X}:=X\times[0,1]$ the metric space equipped with the distance 
    \[\rho((x,s),(y,t))=d(x,y)+|s-t|,\quad \text{for all }x,y\in X,~s,t\in [0,1].\]
    By Proposition~\ref{proposition:E0g->E01 prod}, $\widetilde{X}$ has also the $E_{1,0}$-property.
    Let $\Omega\subset X$ be a nontrivial open subset, $\ell\in C_b(\Omega)$ and $g\in C(\partial \Omega)$ such that $\inf\ell>0$ and the pair $(\ell,g)$ satisfies~\eqref{CC}. Consider the equation

    \[\begin{cases}
        s[u](x)=\ell(x),\quad&\text{for all }x\in\Omega,\\
        u=g(x),\quad &\text{for all }x\in\partial \Omega.
    \end{cases}\]
    Let $V:\overline{\Omega}\to\R$ be the function defined by
    \[
    V(x):=\inf\left \{\int_0^T \ell(\gamma(t))dt+g(\gamma(T))\,\colon \begin{array}{l}
     ~ T>0,~\gamma\in \mathrm{Lip}_1([0,T],\overline{\Omega}),\\
     \gamma(0)=x,~\gamma(T)\in \partial \Omega,~\gamma((0,T))\subset\Omega
     \end{array}\right\}.
     \]
     Thanks to Lemma~\ref{lemma:ContinuityV} we know that $V$ is well defined and continuous.
    Let $\overline{x}\in \Omega$. By compactness of $\overline{\Omega}$, the infimum defining $V(\overline{x})$ is in fact a minimum, and thus, it follows that $S[V](\overline{x})\geq \ell(\overline{x})$. 
    Fix $\varepsilon>0$ such that $B(\overline{x},\varepsilon)\subset \Omega$.
    We show that \[s[V](\overline{x})\leq K_\varepsilon:=\sup_{B(\overline{x},\varepsilon)}\ell.\]
    
    Note that $K_{\varepsilon}>0$. Reasoning towards a contradiction, assume that $s[V](\overline{x})>K_\varepsilon$. Therefore, there is ${\{x_k\}_k\subset B(\overline{x},\varepsilon)}$, with $\lim_k x_k=\overline{x}$, such that 
    \begin{align}\label{bad sequence}
      s[V](\overline{x})\geq\lim_{k\to\infty} \dfrac{V(\overline{x})-V(x_k)}{d(\overline{x},x_k)} > K_\varepsilon.  
    \end{align}
    Denote by $x_\infty:=\overline{x}$.
    Up to a subsequence and since $d_I$ is continuous, we can assume that for any $k,l\in \N\cup\{\infty\}$, there is a $1$-Lipschitz curve $\gamma_{k,l}:[0,d_I(x_k,x_l)]\to B(\overline{x},\varepsilon)$ such that $\gamma_{k,l}(0)=x_k$ and $\gamma_{k,l}(d_I(x_k,x_l))=x_l$.
    So, by  definition of $V$, we deduce that
    \begin{align}\label{a future CC}
        |V(x_k)-V(x_l)|\leq \int_{0}^{d_I(x_k,x_l)}\ell(\gamma_{k,l}(t))dt\leq K_{\varepsilon}d_I(x_k,x_l),\quad\text{for all }k,l\in \N\cup\{\infty\}.
    \end{align}
    Consider now the set
    \[\widetilde{S}:=\{(x_k,0):k\in \N\cup\{\infty\}\}\subset \widetilde{X}\]
    and the boundary condition $h(x_k,0) = V(x_k)/K_{\varepsilon}$. Then, we can set the equation
    \[\begin{cases}
        s[w](x,t)=1,\quad&\text{for all}~(x,t)\in\widetilde{X}\setminus\widetilde{S}, \\
        w(x,t)=h(x,t) := V(x)/K_{\varepsilon},\quad &\text{for all }(x,t)\in\widetilde{S}.
    \end{cases}\]
    Since $\widetilde{X}$ is a compact and has the $E_{1,0}$ property, it has the $E_{1,g}$ property as well, by Proposition~\ref{prop:E01->E0g}. Therefore, the above equation admits a continuous solution $W:\widetilde{X}\to \R$. Indeed, \eqref{a future CC} implies that the pair $(1,h)$ satisfies $(CC)$. Moreover, we have that
    \begin{align}\label{infimum in the product}
        W(x,t)=\inf_{j\in \N\cup\{\infty\}}\{V(x_j)/K_{\varepsilon} +\underbrace{d_I(x,x_j)+ |t-0|}_{\rho_I((x,t),(x_j,0))} \}.
    \end{align}
    Due to \eqref{a future CC} and \eqref{infimum in the product}, we deduce that for all $k\in \N\cup\{\infty\}$
    \[W(x_k,1)= V(x_k)/K_{\varepsilon}+1.\]
    That is, for each $k\in \N$ the infimum~\eqref{infimum in the product} is attained at $j=k$.
    However, combining~\eqref{bad sequence} and the fact that $W$ has local slope $1$, we get
    \begin{align*}
    K_\varepsilon&<\lim_{k\to\infty} \dfrac{V(\overline{x})-V(x_k)}{d(\overline{x},x_k)}\\
    &=\lim_{k\to\infty} \dfrac{K_{\varepsilon}(W(\overline{x},1)-W(x_k,1))}{\rho((\overline{x},1),(x_k,1))}\leq K_\varepsilon,
    \end{align*}
    which is a contradiction. Therefore, $s[V](\overline{x})\leq K_\varepsilon$. 
    Since $\varepsilon>0$ can be chosen arbitrarily small and $\ell $ is continuous, we deduce that
    \[s[V](\overline{x})\leq \inf_{\varepsilon>0}K_\varepsilon = \ell(\overline{x}).\]
    
    The proof is now complete.
\end{proof}

\section{Comparison with other notions of solutions}\label{sec:Comparison}

Recall from \cite{Liu2021:Equivalence} that when $(X,d)$ is a length space, then the solutions of \eqref{eq: eikonal complete} coincide with the viscosity solutions in the sense of Giga, Hamamuki and Nakayasu \cite{Giga2015:Eikonal} and in the sense Gangbo and Święch \cite{Gangbo2015:Metric,Gangbo2014:Optimal}.\\ 

Even though the notion of viscosity solution of Gangbo and Święch requires the ambient space to be length, the viscosity solution of Giga, Hamamuki and Nakayasu can be defined in any rectifiably connected space, and it is always given by the optimal control formula~\eqref{eq:OptimalControl-formula}. Moreover, as proven in \cite{Liu2021:Equivalence}, if the intrinsic distance of $(X,d)$ is continuous then the viscosity solution of Giga, Hamamuki and Nakayasu coincides with the viscosity solution of Gangbo and Święch in $(X,d_I)$ and with the pointwise solution of
\begin{align}\label{main equation - intrinsic}
    \begin{cases}
        s_I[u](x)=\ell(x),\quad &\text{for all }x\in\Omega,\\
        u(x)=g(x),\quad&\text{for all }x\in \partial\Omega.
    \end{cases}
\end{align}
where $s_I[u](x)$ is the slope of $u$ at $x$ using the intrinsic distance $d_I$ instead of $d$, see~\eqref{eq:Def-Slope}.
However, if the space $(X,d)$ is not length, the solution $u$ of~\eqref{main equation - intrinsic} may fail to have first-order compatibility with respect to the original metric, in the sense that $s[u](x)$ might differ from $\ell(x)$ (it always holds that $u$ will be a super solution, that is $s[u](x)\geq \ell(x)$). One could interpret, in the context of compact metric spaces, the eikonal spaces exactly as those where the solution $u$ of~\eqref{main equation - intrinsic} verify $s[u](x) = s_I[u](x) = \ell(x)$ for all $x\in \Omega$. A summary is depicted in Figure~\ref{fig:scheme-spaces}:
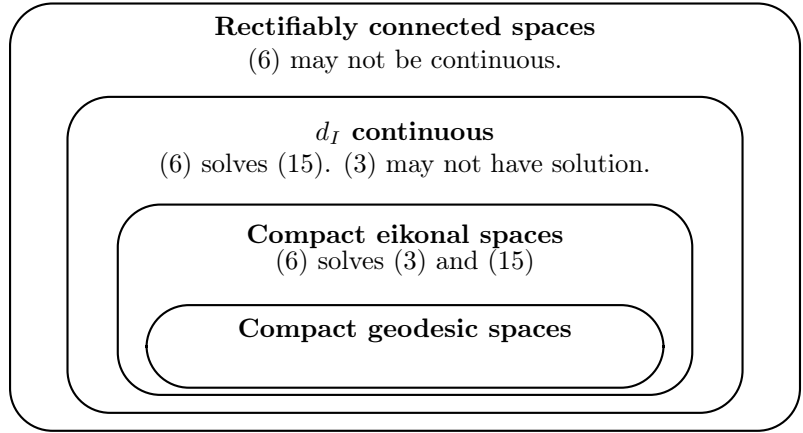
\begin{figure}[ht]
\centering
\begin{tikzpicture}[scale=0.95, every node/.style={font=\small}]
    \draw[thick, rounded corners=16pt] (0,0.85) rectangle (11,6.8);
    \node at (5.5,6.45) {\textbf{Rectifiably connected spaces}};
    \node at (5.5,6.0) {\eqref{eq:OptimalControl-formula} may not be continuous.};

    \draw[thick, rounded corners=16pt] (0.8,1.1) rectangle (10.2,5.5);
    \node at (5.5,5.0) {\textbf{$d_I$ continuous}};
    \node at (5.5,4.55) {\eqref{eq:OptimalControl-formula} solves \eqref{main equation - intrinsic}. \eqref{eq: eikonal complete} may not have solution.};

    \draw[thick, rounded corners=16pt] (1.5,1.35) rectangle (9.5,4);
    \node at (5.5,3.55) {\textbf{Compact eikonal spaces}};
    \node at (5.5,3.2) {\eqref{eq:OptimalControl-formula} solves \eqref{eq: eikonal complete} and \eqref{main equation - intrinsic}};

    \draw[thick, rounded corners=16pt] (1.9,1.45) rectangle (9.1,2.6);
    \node at (5.5,2.25) {\textbf{Compact geodesic spaces}};
\end{tikzpicture}
\caption{Classification of compact metric spaces and notions of solutions.}\label{fig:scheme-spaces}
\end{figure}

\begin{enumerate}
    \item In rectifiably connected spaces, we use the notion of solution of Giga, Hamamuki and Nakayasu, which is given by the formula~\eqref{eq:OptimalControl-formula}, see \cite{Giga2015:Eikonal}. This solution may fail to be continuous, and it is only a super solution of~\eqref{eq: eikonal complete}.
    \item If $d_I$ is continuous, then solution of Giga, Hamamuki and Nakayasu, coincides with the pointwise solution of~\eqref{main equation - intrinsic}, see \cite{Liu2021:Equivalence}. The solution is continuous, but it is only a super solution of~\eqref{eq: eikonal complete}.
    \item In the compact setting, if $(X,d)$ is an eikonal space, then the pointwise solution of~\eqref{main equation - intrinsic} and the pointwise solution of~\eqref{eq: eikonal complete} coincide.
\end{enumerate}

If the metric space is not compact, the classification gets more complicated. Length spaces remain eikonal, but eikonal spaces may not even be connected, and are therefore not included in neither the class of spaces where $d_I$ is continuous nor the class of rectifiably connected spaces.
\begin{example}
Consider two parallel copies of $\R$ in $\R^2$, let us say $L_1 = \R\times\{0\}$ and $L_2 = \R\times \{1\}$, and set $X = L_1\cup L_2$, endowed with the induced distance from $\R^2$. This space is clearly not rectifiably connected. However, every eikonal equation of the form $\eqref{eq: eikonal complete}$ admits solutions. Indeed, take $\Omega\subset X$ a proper open set, a continuous function $\ell:\Omega\to (0,+\infty)$ and a continuous function $g:\partial \Omega\to \R$ verifying~\eqref{CC}. Note that in this space $\partial \Omega$ could be empty. In this case, we assume that~\eqref{CC} holds by vacuity.  We distinguish two cases:
\begin{enumerate}
    \item $L_1\not\subset \Omega$ and $L_2\not\subset \Omega$: In this case, we just solve two independent equations with $\Omega_1 = L_1\cap \Omega$ and $\Omega_2 = L_2\cap \Omega$.
    \item $L_1\subset \Omega$: In this case, since $\Omega$ must be a proper subset of $X$, we solve the equation in $\Omega_2 = L_2\cap \Omega$, obtaining the solution $u_2:\overline{\Omega}_2\to\R$, and define
    \[
    u:x = (x_1,x_2)\mapsto u(x) = \begin{cases}
        \int_0^{x_1} \ell(t,0)dt + c\quad &\text{ if }x\in L_1,\\
        u_2(x) \quad&\text{ if }x\in \overline{\Omega_2},
    \end{cases}
    \]
    where $c\in\R$ is a constant. Note that in this case there are infinitely many solutions. If $\partial \Omega = \emptyset$, that is $\Omega = L_1$, we simply set $u(x) = \int_0^{x_1} \ell(t,0)dt + c$. The case when $L_2\subset \Omega$ is analogous.
\end{enumerate}
In any case, the eikonal equation admits at least one solution, and so this is an eikonal space.\hfill$\Diamond$
\end{example}

\paragraph{Open questions:} 
In this manuscript we have studied in depth the slope eikonal equations in compact spaces. 
From this, we would like to point out the following two interesting lines of research: the study and classification of noncompact eikonal spaces, as well as, the study of well-posedness of more complicated Hamilton-Jacobi equations.  On the other hand, one can observe from our proof in Section~\ref{sec:main} that Corollary~\ref{cor:intrinsic-distance-to-K} can be improved by reducing the family of closed sets only to countable ones: a compact metric space $(X,d)$ is eikonal if and only if $d_I(\cdot,K)$ has slope $1$ on $X\setminus K$, for every $K\subset X$ closed and countable. 
We conjecture that this can be further reduced to singletons.\\

\textbf{Question 1:} Let $(X,d)$ be a compact metric space such that $d_I(\cdot,x)$ has slope 1 on $X\setminus\{x\}$ for every $x\in X$. Is $(X,d)$ an eikonal space?
\bibliographystyle{abbrv}
\bibliography{biblio}

	\newpage
	
	\rule{5 cm}{0.5 mm} \bigskip
    \newline\noindent David SALAS, Francisco VENEGAS M.
	
	\medskip
	
	\noindent Instituto de Ciencias de la Ingenier\'{i}a, Universidad de
	O'Higgins\newline Av. Libertador Bernardo O'Higgins 611, Rancagua, Chile
	\smallskip
	
	\noindent E-mail: \texttt{david.salas@uoh.cl}, \texttt{francisco.venegas@uoh.cl} \newline\noindent
	\texttt{http://davidsalasvidela.cl}, \texttt{https://sites.google.com/view/francisco-venegas-m}
	\medskip
	
	\noindent Research supported by the grant: \smallskip\newline CMM 
	FB210005 BASAL, FONDECYT 1251159,  FONDECYT 3250857 (Chile).\newline\vspace{0.2cm}

    \noindent Sebasti{\'a}n TAPIA-GARC{\'I}A
	
	\medskip
	
	\noindent Institute of Statistics and Mathematical Methods in Economics,
	E105-04 \newline TU Wien, Wiedner Hauptstra{\ss}e 8, A-1040
	Wien\smallskip\newline\noindent E-mail: \texttt{
		sebastian.tapia.garcia@tuwien.ac.at}\newline\noindent
	\texttt{https://sites.google.com/view/sebastian-tapia-garcia}
	\medskip
	
	\noindent Research partially supported by the Austrian Science Fund grant \textsc{FWF
		P-36344N}.\newline\vspace{0.2cm}

\end{document}